\RequirePackage[l2tabu, orthodox]{nag}
\documentclass[a4paper]{amsart} 
\pdfoutput=1
\usepackage{scrhack}
\usepackage[T1]{fontenc} 
\usepackage[utf8]{inputenc}
\usepackage[english]{babel}
\usepackage[babel]{csquotes}
\usepackage{lmodern} 
\usepackage[stretch=10]{microtype}
\usepackage{pdflscape}
\usepackage{amssymb}
\usepackage{mathrsfs}
\usepackage{amsthm}
\usepackage{mathtools}
\usepackage{graphicx}
\usepackage{fullpage}
\usepackage{color}

\usepackage{fixltx2e} 
\usepackage[pdftex,
            pdfauthor={Jan Hesse},
            pdftitle={},
             breaklinks=true]{hyperref}

\newcommand{\Z}{\mathbb{Z}}

\newcommand{\K}{\mathbb{K}}
\newcommand{\F}{\mathbb{F}}

\DeclareSymbolFont{bbold}{U}{bbold}{m}{n}
\DeclareSymbolFontAlphabet{\mathbbold}{bbold}

\newcommand{\sK}{\mathscr{K}}

\newcommand{\cC}{\mathcal{C}}
\newcommand{\cD}{\mathcal{D}}

\newcommand{\core}[1]{\sK ({#1})}

\DeclareMathOperator{\Hom}{Hom}

\DeclareMathOperator{\Aut}{Aut}

\DeclareMathOperator{\SymMonBicat}{SymMonBicat}
\DeclareMathOperator{\Cob}{Cob}
\DeclareMathOperator{\Fun}{Fun}

\DeclareMathOperator{\Frob}{Frob}

\DeclareMathOperator{\Vect}{Vect}

\DeclareMathOperator{\ev}{ev}
\DeclareMathOperator{\coev}{coev}
\DeclareMathOperator{\End}{End}
\DeclareMathOperator{\Alg}{Alg}

\DeclareMathOperator{\Nat}{Nat}

\DeclareMathOperator{\Bicat}{Bicat}

\DeclareMathOperator{\Pic}{Pic}

\newcommand{\ori}{\mathrm{or}}

\newcommand{\id}{\mathrm{id}}
\newcommand{\Id}{\mathrm{Id}}

\newcommand{\ot}{\otimes}
\newcommand{\eps}{\varepsilon}
\newcommand{\To}{\Rightarrow}

\newcommand{\fd}{\mathrm{fd}}
 
\newcommand{\fr}{\mathrm{fr}}

\theoremstyle{plain} 
\newtheorem{newdef}{Definition}[section]
\newtheorem{theorem}[newdef]{Theorem}
\newtheorem{lemma}[newdef]{Lemma}
\newtheorem{prop}[newdef]{Proposition}
\newtheorem{cor}[newdef]{Corollary}
\newtheorem{conj}[newdef]{Conjecture}

\theoremstyle{definition}
\newtheorem{remark}[newdef]{Remark}
\newtheorem{ex}[newdef]{Example}

\numberwithin{equation}{section}

\begin{document}
 \title{The Serre Automorphism via Homotopy Actions
  and the Cobordism Hypothesis for Oriented Manifolds} 
 \author{Jan
  Hesse} 
 \address{Fachbereich Mathematik,
 Universit\"at Hamburg,
  Bereich Algebra und Zahlentheorie,
 Bundesstraße 55, D – 20 146
  Hamburg}
 \author{Alessandro Valentino}
 \address{Institut f\"ur
  Mathematik,
 Universit\"at Z\"urich
 Winterthurerstrasse 190
 CH-8057
  Z\"urich
 }
 \date{}
 
\begin{abstract}
  We explicitly construct an $SO(2)$-action on a skeletal version of
  the 2-dimensional framed bordism bicategory. By the 2-dimensional
  Cobordism Hypothesis for framed manifolds, we obtain an
  $SO(2)$-action on the core of fully-dualizable objects of the target
  bicategory. This action is shown to coincide with the one given by
  the Serre automorphism. We give an explicit description of the
  bicategory of homotopy fixed points of this action, and discuss its
  relation to the classification of oriented 2d topological quantum field
  theories. 
\end{abstract}
\begin{flushright}
    {ZMP-HH/16-29 \\ Hamburger Beitr\"age zur Mathematik Nr. 636}\\[1cm]
\end{flushright}
\maketitle
\section{Introduction}
As defined by Atiyah in \cite{ati88} and Segal in \cite{segal}, an
$n$-dimensional Topological Quantum Field Theory (TQFT) consists of a
functor between two symmetric monoidal categories, namely a category
of $n$-cobordisms, and a category of algebraic objects. This
definition was introduced to axiomatize the locality properties of
the path integral, and has given rise to a fruitful interplay between
mathematics and physics in the last 30 years. A prominent example is
given by a quantum-field-theoretic interpretation of the Jones
polynomial by Witten in \cite{witten89}.

More recently, there has been a renewed interest in the study of
TQFTs, due in great part to the Baez-Dolan Cobordism Hypothesis and
its proof by Lurie, whose main objects of investigation are
\emph{fully extended} TQFTs. These are a generalization of the notion
of $n$-dimensional TQFTs, where data is assigned to manifolds of
codimension up to $n$. The Baez-Dolan Cobordism Hypothesis, originally
stated in \cite{bd95}, and proved by Lurie in \cite{Lurie09} in an
$\infty$-categorical version, can be stated as follows: fully extended
\emph{framed} TQFTs are classified by their value on a point, which
must be a fully dualizable object in the target symmetric monoidal
$ (\infty,n)$-category $\cC$. Moreover, the $\infty$-groupoid
$\core{\cC^{\fd}}$ given by the core of fully dualizable objects of
$\cC$ carries a homotopy $O(n)$-action induced by the ``rotation of
the framing'' on the framed $(\infty,n)$-cobordism category
\cite[Corollary 2.4.10]{Lurie09}. The inclusion
$SO(n)\hookrightarrow O(n)$ then induces an $SO(n)$-action on
$\core{\cC^{\fd}}$.
By the Cobordism Hyothesis for manifolds whose tangent bundle is equipped with an additional $G$-structure, homotopy fixed-points for this action classify fully extended
\emph{oriented} TQFTs. It is relevant to notice that in \cite{Lurie09}
the homotopy $O(n)$-action on the framed $(\infty,n)$-category of
cobordisms is not explicitly constructed, or even briefly
sketched. For an extensive introduction to extended TQFTs and the
Cobordism Hypothesis, we refer the reader to \cite{freedcob}.

Blurring the distinction between $(\infty,2)$-categories and
bicategories, in \cite{fhlt} it is argued that in the case where the
target is given by the bicategory $\Alg_2$ of algebras, bimodules, and
intertwiners, the fully dualizable objects are semisimple
finite-dimensional algebras, and that the additional
$SO(2)$-fixed-points structure should correspond to the structure of a
symmetric Frobenius algebra. Via a direct construction, in
\cite{schommerpries-classification} it is showed that the bigroupoid
$\Frob$ of Frobenius algebras, Morita contexts and intertwiners indeed
classifies fully extended oriented 2-dimensional TQFTs valued in
$\Alg_2$. In \cite{davi11}, it is observed that the $SO(2)$-action
given by the Serre automorphism on the core of fully-dualizable
objects of $\Alg_2$ is trivializable. In a purely bicategorical
setting, in \cite{hsv16} the homotopy-fixed-point bigroupoid of the
$SO(2)$-action on $\Alg_2$ is computed, and it is shown that it
coincides with $\Frob$.

In the present paper we provide an explicit $SO(2)$-action on the
framed bordism bicategory, and show that the $SO(2)$-action induced on
$\core{\cC^{\fd}}$ for any symmetric monoidal bicategory $\cC$ is
given by the Serre automorphism, regarded as a pseudo-natural isomorphism of
the identity functor. More precisely, we make use of a presentation
of the framed bordism bicategory provided in \cite{piotr14} to
construct such an $SO(2)$-action.

By the Cobordism Hypothesis for framed manifolds, which has been
proven in the setting of bicategories in \cite{piotr14}, there is an
equivalence of bicategories
\begin{equation}
  \label{eq:framed-cob-hyp}
  \Fun_\ot(\Cob_{2,1,0}^\fr, \cC) \cong \core{\cC^\fd}.
\end{equation}
This equivalence allows us to transport the $SO(2)$-action on the
framed bordism bicategory to the core of fully-dualizable objects of
$\cC$. We then prove that this induced $SO(2)$-action on
$\core{\cC^\fd}$ is given precisely by the Serre automorphism, showing
that the Serre automorphism has indeed a geometric origin, as expected
from \cite{Lurie09}.

Along the way, we also provide results concerning monoidal homotopy
actions which are useful in determining when such actions are
trivializable. The relevance for TQFT is the following: in the case of
a trivializable $SO(2)$-action, \emph{any} framed fully extended 2d
TQFT can be promoted to an oriented one by providing the appropriate
structure of a homotopy fixed point. In particular, we apply these results to the case of
invertible 2d TQFTs, which have recently attracted interest for their
application to condensed matter physics, more specifically to the
study of topological insulators \cite{Freed:2014iua, Freed:2014eja,
  freedhopkins}. Namely, fully extended invertible TQFTs have been
proposed as the low energy limit of short-range entanglement systems;
see \cite{Freed:2014eja} for a discussion of these topics.

First defintions of monoidal bicategories appear in \cite{kv-bicat}, \cite{baez-neuchl} and \cite{day-street}, with a first full definition of a symmetric monoidal bicategory in \cite{mccrudden}. We will refer to \cite{schommerpries-classification} for technical details.
In section \ref{sec:fram-bord-bicat}, we use the
wire-diagram calculus developed in \cite{bart14}.

It is worth noticing that the study of actions of groups on higher
categories and their homotopy fixed points is also of independent
interest, see for instance \cite{egno, bermomb} for the case of finite
groups.


The paper is organized as follows.\\ 
In Section
\ref{sec:fully-dual-objects} we recall the notion of a fully-dualizable
object in a symmetric monoidal bicategory $\cC$. For each such an
object $X$, we define the Serre automorphism as a certain
1-endomorphism of $X$. We show that the Serre automorphism is a
pseudo-natural transformation of the identity functor on
$\core{\cC^\fd}$, which is moreover monoidal. This suffices to define an
$SO(2)$-action on $\core{\cC^\fd}$.

Section \ref{sec:triviality-actions} investigates when a group action
on a bicategory $\cC$ is equivalent to the trivial action. We obtain a
general criterion for when such an action is trivializable. 

In Section \ref{sec:comp-homot-fixed}, we compute the bicategory of
homotopy fixed points of an $SO(2)$-action coming from a
pseudo-natural transformation of the identity functor of an arbitrary
bicategory $\cC$. This generalizes the main result in \cite{hsv16},
which computes homotopy fixed points of the trivial $SO(2)$-action on
$\Alg_2^\fd$. Our more general theorem allows us to give an explicit
description of the bicategory of homotopy fixed points of the Serre
automorphism.

In Section \ref{sec:fram-bord-bicat}, we introduce a skeletal version
of the framed bordism bicategory by generators and relations, and
define a non-trivial $SO(2)$-action on this bicategory. By the framed
Cobordism Hypothesis, as in Equation (\ref{eq:framed-cob-hyp}), we
obtain an $SO(2)$-action on $\core{\cC^\fd}$, which we prove to
coincide with the one given by the Serre automorphism.

In Section \ref{sec:inv} we discuss invertible 2d TQFTs, providing a
general criterion for the trivialization of the $SO(2)$-action in this
case.

In Section \ref{sec:comments}, we give an outlook on \emph{homotopy
  co-invariants} of the $SO(2)$-action, and argue about their relation
to the Cobordism Hypothesis for oriented manifolds.

\section*{Acknowledgments}
The authors would like to thank Domenico Fiorenza, Claudia
Scheimbauer and Christoph Schweigert for useful
discussions. Furthermore, the authors would like to thank the referee
for useful comments and remarks which improved the paper significantly. J.H. is supported by the RTG 1670
\enquote{Mathematics inspired by String Theory and Quantum Field
  Theory}. A.V. is partly supported by the NCCR SwissMAP, funded by
the Swiss National Science Foundation, and by the COST Action MP1405
QSPACE, supported by COST (European Cooperation in Science and
Technology).
\section{Fully-dualizable objects and the Serre automorphism}
\label{sec:fully-dual-objects}
The aim of this section is to introduce the main objects of the present paper. On
the algebraic side, these are fully-dualizable objects in a symmetric
monoidal bicategory $\cC$, and the Serre automorphism. Though some of the
following material has already appeared in the literature, we recall
the relevant definitions in order to fix notation. For details, we refer the reader to \cite{piotr14}.
\begin{newdef}
  \label{def:dual-pair}
  A dual pair in a symmetric monoidal bicategory $\cC$ consists of an
  object $X$, an object $X^*$, two 1-morphisms
  \begin{equation}
    \begin{aligned}
      \ev_X &: X \ot X^* \to 1 \\
      \coev_X &: 1 \to X^* \ot X
    \end{aligned}
  \end{equation}
  and two invertible 2-morphisms $\alpha$ and $\beta$ in the diagrams
  below.
  \begin{equation}
    \vcenter{\hbox{\includegraphics[]{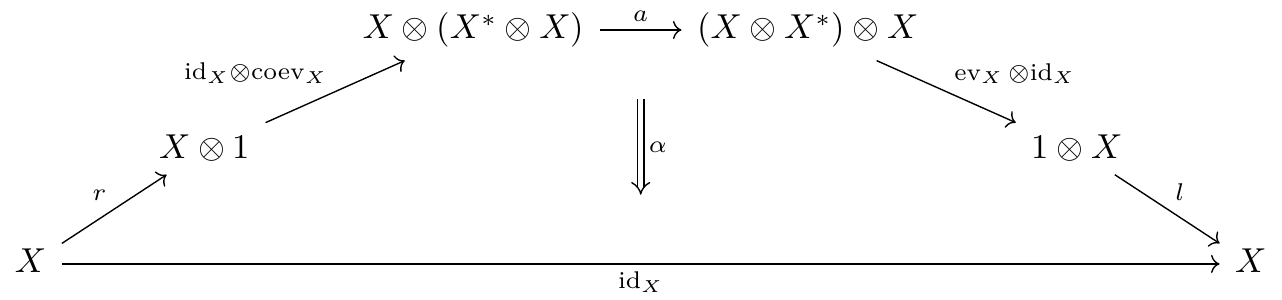}}}
  \end{equation}
  \begin{equation}
    \vcenter{\hbox{\includegraphics[]{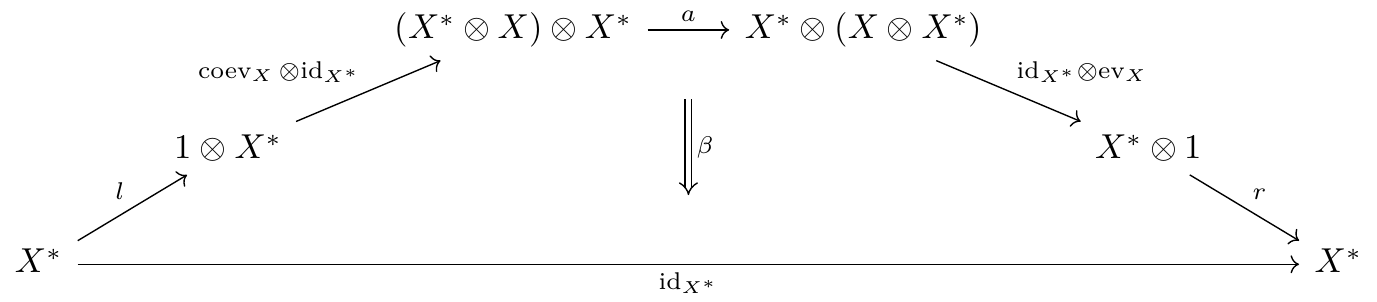}}}
  \end{equation}
  We call an object $X$ of $\cC$ dualizable if it can be completed to
  a dual pair. A dual pair is said to be coherent if the
  ``swallowtail'' equations are satisfied, as in
  \cite[Def. 2.6]{piotr14}.
\end{newdef}

\begin{remark}
  Given a dual pair, it is always possible to modify the 2-cell
  $\beta$ in such a way that the swallowtail are fulfilled,
  cf. \cite[Theorem 2.7]{piotr14}.
\end{remark}
Dual pairs can be organized into a bicategory by defining appropriate
1- and 2-morphisms between them, cf. \cite[Section
2.1]{piotr14}. The bicategory of dual pairs turns
out to be a 2-groupoid. Moreover, the bicategory of coherent dual
pairs is equivalent to the core of dualizable objects in $\cC$. In
particular, this shows that any two coherent dual pairs over the same
dualizable object are equivalent.

We now come to the stronger concept of fully-dualizability.
\begin{newdef}
  \label{def:fully-dualizable-object}
  An object $X$ in a symmetric monoidal bicategory is called
  fully-dualizable if it can be completed into a dual pair and the
  evaluation and coevaluation maps admit both left- and right
  adjoints.
\end{newdef}
Note that if left- and right adjoints exists, the adjoint maps will
have adjoints themselves, since we work in a bicategorical setting,
cf. \cite{piotr14} 
. Note that if left- and right adjoints for the 1-morphisms $\ev$ and $\coev$ exist, these adjoint 1-morphisms will in turn have additional adjoints themselves. Thus, Definition
\ref{def:fully-dualizable-object} agrees with the definition of
\cite{Lurie09} in the special case of bicategories.

\subsection{The Serre automorphism}
Recall that by definition, the evaluation morphism for a fully
dualizable object $X$ admits both a right-adjoint $\ev^R_X$ and a left
adjoint $\ev^L_X$. We use these adjoints to define the
Serre-automorphism of $X$:
\begin{newdef}
  \label{def:serre-auto}
  Let $X$ be a fully-dualizable object in a symmetric monoidal
  bicategory. The Serre automorphism of $X$ is the following
  composition of 1-morphisms:
  \begin{equation}
    S_X:  X \cong X \ot 1 \xrightarrow{ \id_X \ot \ev^R_X} X \ot X \ot
    X^* \xrightarrow{\tau_{X,X} \ot \id_{X^*}} X \ot X \ot X^*
    \xrightarrow{\id_X \ot \ev_X} X \ot 1 \cong X.
  \end{equation}
\end{newdef}
Notice that the Serre automorphism is actually a 1-equivalence of $X$,
since an inverse is given by the 1-morphism
\begin{equation}
  \label{eq:serre-inverse}
  S_X^{-1}=(\id_X \circ \ev_X) \circ (\tau_{X,X} \ot \id_{X^*}) \circ
  (\id_X \ot \ev_X^L),
\end{equation}
cf. \cite{Lurie09,dss13}.\\
The next lemma is well-known \cite{Lurie09, piotr14}, and is
straightforward to show graphically.
\begin{lemma}
  \label{lem:right-adjoint-ev-serre}
  Let $X$ be fully-dualizable in $\cC$. Then, there are 2-isomorphisms
  \begin{equation}
    \begin{aligned}
      \ev_X^R &\cong \tau_{X^*,X} \circ (\id_{X^*} \ot S_X) \circ \coev_X \\
      \ev_X^L &\cong \tau_{X^*,X} \circ (\id_{X^*} \ot S_X^{-1}) \circ
      \coev_X .
    \end{aligned}
  \end{equation}
\end{lemma}

Next, we show that the Serre automorphism is actually a pseudo-natural transformation of the
identity functor on the maximal subgroupoid of $\cC$, as suggested in \cite{schommer-pries2013-dual}. To the
best of our knowledge, a proof of this statement has not appeared in
the literature so far, hence we illustrate the details in the
following. We begin by showing that the evaluation 1-morphism is
\enquote{dinatural}.

\begin{lemma}
  \label{lem:ev-dinatural}
Let $X$ be dualizable in $\cC$. The evaluation 1-morphism $\ev_X$ is
  \enquote{dinatural}: for every 1-morphism $f:X \to Y$ between
  dualizable objects, there is a natural 2-isomorphism $\ev_f$ in the
  diagram below.
\begin{equation}
    \vcenter{\hbox{\includegraphics{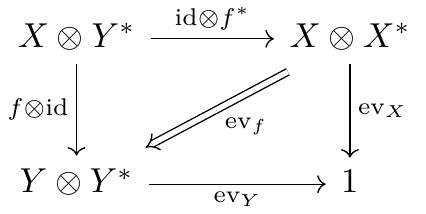}}}
  \end{equation}
By \enquote{di-naturality}, we explicitly mean that for every 2-morphism $\alpha:f \To g$ in $\cC$, the following diagram commutes
  \begin{equation}
  \label{eq:ev-f-dinatural}
   \vcenter{\hbox{\includegraphics[width=0.9\textwidth]{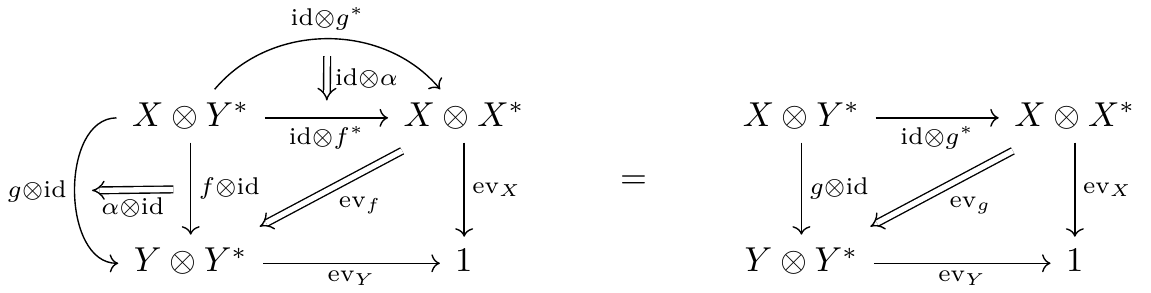}}}
  \end{equation}
    \end{lemma}
\begin{proof}
  We explicitly write out the definition of $f^*$ and define $\ev_f$
  to be the composition of the 2-morphisms in the diagram below.
  \begin{equation}
    \ev_f:= 
    \vcenter{\hbox{\includegraphics[width=0.9\textwidth]{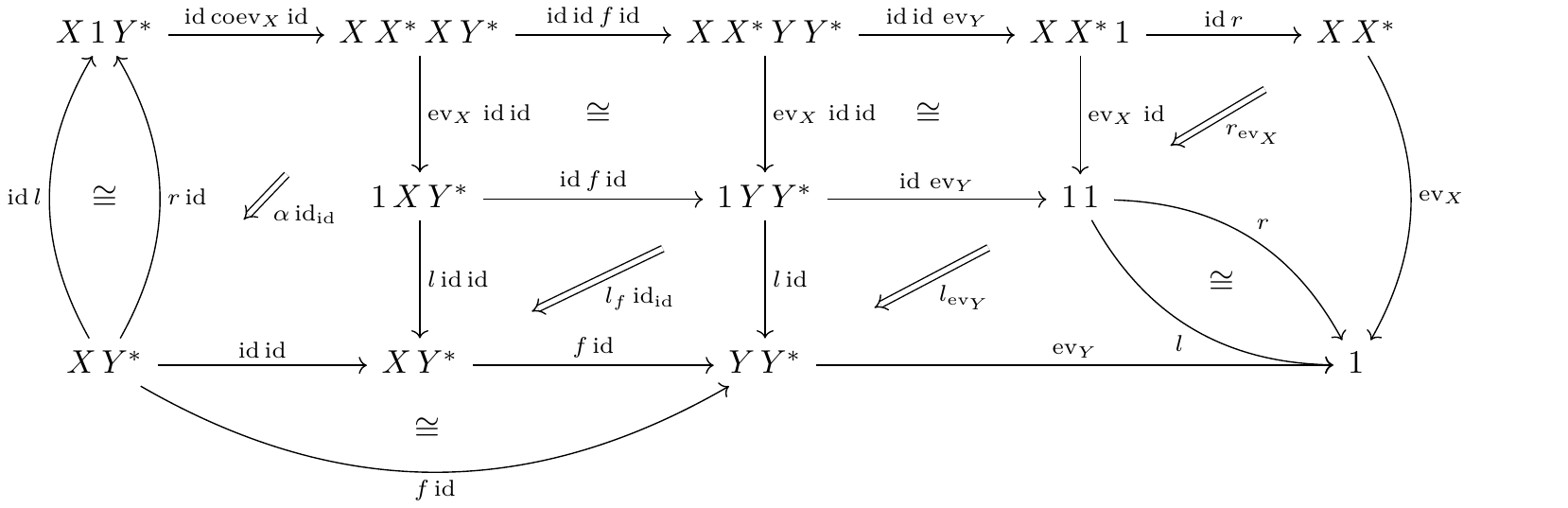}}}
  \end{equation}
  Since the 2-morphism $\ev_f$ is given by the composition of
  associators and unitors which are natural 2-morphisms, it is natural itself, and thus the diagram in equation \ref{eq:ev-f-dinatural} commutes.
\end{proof}
In order to show that the Serre automorphism is pseudo-natural, we also need to
show the dinaturality of the right adjoint of the evaluation.
\begin{lemma}
  \label{lem:right-adjoint-dinatural}
  For a fully-dualizable object $X$ of $\cC$, the right adjoint
  $\ev^R$ of the evaluation is \enquote{dinatural} with respect to
  1-equivalences: for every 1-equivalence $f:X \to Y$ between
  fully-dualizable objects, there is a natural 2-isomorphism $\ev^R_f$
  in the diagram below.
\begin{equation}
  \vcenter{\hbox{\includegraphics{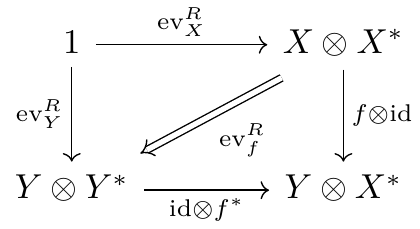}}}
\end{equation}
\end{lemma}
\begin{proof}
  In a first step, we show that $f \ot {(f^*)}^{-1} \circ \ev_X^R $ is
  a right-adjoint to $\ev_X \circ ( f^{-1} \ot f^*)$. In formula:
  \begin{equation}
    \label{eq:right-adjoint-dinatural}
    (  \ev_X \circ f^{-1} \ot f^*)^R= f \ot {(f^*)}^{-1}  \circ \ev_X^R.
  \end{equation}
  Indeed, let
  \begin{equation}
    \begin{aligned}
      \eta_X: \id_{X \ot X^*} & \to \ev_X^R \circ \ev_X \\
      \eps_X: \ev_X \circ \ev^R_X & \to \id_1
    \end{aligned}
  \end{equation}
  be the unit and counit of the right-adjunction of $\ev_X$ and its
  right adjoint $\ev_X^R$. We construct unit and counit for the
  adjunction in Equation \eqref{eq:right-adjoint-dinatural}. Let
  \begin{equation}
    \begin{aligned}
      \tilde \eps &: \ev_X \circ( f^{-1} \ot f^*) \circ( f \ot
      {(f^*)}^{-1} )\circ \ev_X^R \cong \ev_X \circ
      \ev_X^R \xrightarrow{\eps_X} \id_1  \\
      \tilde \eta& :\id_{Y \ot Y^*} \cong ( f \ot {(f^*)}^{-1}) \circ
      (f^{-1} \ot f^*) \xrightarrow{ \id * \eta_X * \id}(f \ot
      {(f^*)}^{-1})\circ \ev_X^R \circ \ev_X \circ( f^{-1} \ot f^*).
    \end{aligned}
  \end{equation}
  Now, one checks that the quadruple
  \begin{equation}
    ( \ev_X \circ (f^{-1} \ot f^*), \,  (f \ot {(f^*)}^{-1} ) \circ
    \ev_X^R , \, \tilde   \eps,  \tilde \eta) 
  \end{equation}
  fulfills indeed the axioms of an adjunction. This follows from the
  fact that the quadruple $(\ev_X , \ev_X^R ,\eps_X , \eta_X)$ is an
  adjunction. This shows Equation \eqref{eq:right-adjoint-dinatural}.

  Now, notice that due to the dinaturality of the evaluation in Lemma
  \ref{lem:ev-dinatural}, we have a natural 2-isomorphism
  \begin{equation}
    \ev_Y \cong \ev_X \circ (f^{-1} \ot f^*)  .
  \end{equation}
  Combining this 2-isomorphism with Equation
  \eqref{eq:right-adjoint-dinatural} shows that the right adjoint of
  $\ev_Y$ is given by $f \ot {(f^*)}^{-1} \circ \ev_X^R$.  Since all
  right-adjoints are isomorphic the 1-morphism $f \ot {(f^*)}^{-1}
  \circ \ev_X^R $ is isomorphic to $\ev_Y^R$, as desired.
\end{proof}

We can now prove the following proposition.
\begin{prop}
  \label{prop:serre-pseudo-natural-core-fd}
  Let $\cC$ be a symmetric monoidal bicategory. Denote by $\core{\cC}$
  the maximal sub-bigroupoid of $\cC$. The Serre automorphism $S$ is a
  pseudo-natural isomorphism of the identity functor on $\core{
    \cC^\fd}$.
\end{prop}
\begin{proof}
  Let $f: X \to Y$ be a 1-morphism in $\core{ \cC^\fd}$. We need
  to provide a natural 2-isomorphism in the diagram
  \begin{equation}
    \vcenter{\hbox{\includegraphics{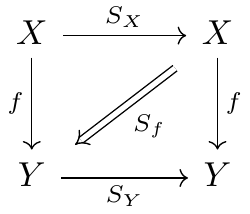}}}
  \end{equation}
  By spelling out the definition of the Serre automorphism, we see
  that this is equivalent to filling the following diagram with
  natural 2-cells:
  \begin{equation}
    \vcenter{\hbox{\includegraphics{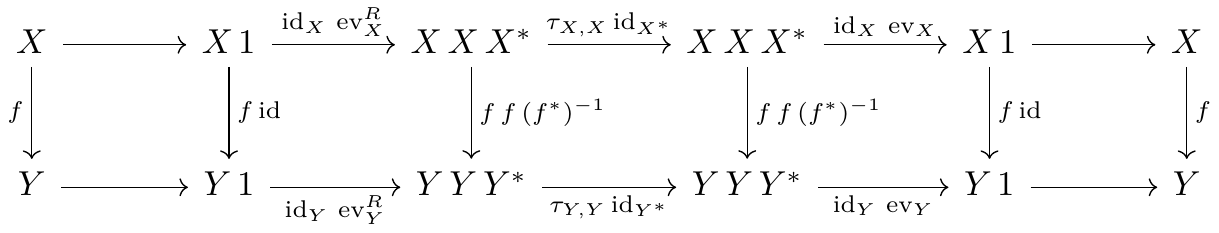}}}
  \end{equation}
  The first, the last and the middle square can be filled with a
  natural 2-cell due to the fact that $\cC$ is a symmetric monoidal
  bicategory. The square involving the evaluation commutes up to a
  2-cell using the mate of the 2-cell of Lemma \ref{lem:ev-dinatural},
  while the square involving the right adjoint of the evaluation
  commutes up a 2-cell using the mate of the 2-cell of Lemma
  \ref{lem:right-adjoint-dinatural}.
\end{proof}

\subsection{Monoidality of the Serre automorphism}
 In this section we
show that the Serre automorphism respects the
 monoidal structure. We
will show that the Serre-automorphism is a \emph{monoidal}
pseudo-natural transformation of the identity functor. We begin with
the following two lemmas:
\begin{lemma}
  \label{lem:tensor-product-dual}
 Let $\cC$ be a monoidal
  bicategory. Let $X$ and $Y$ be dualizable
 objects of $\cC$. Then,
  there is a 1-equivalence
  $ \xi_{X,Y}:(X \ot Y)^*
 \cong Y^* \ot X^*$. Furthermore, this
    1-equivalence $\xi$ is pseudo-natural: suppose that $f:X \to X'$
    and $g:Y \to Y'$ are two 1-morphisms in $\cC$. Then, there is a
    pseudo-natural 2-isomorphism in the diagram in equation
    \eqref{eq:xi-natural}.
  \begin{equation}
   \label{eq:xi-natural}
   \vcenter{\hbox{\includegraphics{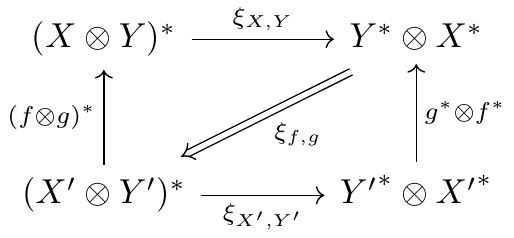}}}
 \end{equation}
\end{lemma}
\begin{proof}
  Define two 1-morphisms in $\cC$ by setting
  \begin{equation}
    ( \id_{Y^*} \ot \id_{X^*} \ot \ev_{X \ot Y}) \circ ( \id_{Y^*} \ot
    \coev_X \ot \id_Y \ot \id_{(X \ot Y)^*}) \circ (\coev_Y \ot \id_{(X
      \ot Y)^*}): (X \ot Y)^* \to Y^* \ot X^*
  \end{equation}
  \begin{equation}
    (\id_{(X \ot Y)^*} \ot \ev_X)  \circ (\id_{(X \ot Y)^*} \ot \id_X \ot
    \ev_Y \ot \id_{X^*})  \circ (\coev_{X \ot Y} \ot \id_Y^* \ot
    \id_{X^*}) : Y^* \ot X^* \to (X \ot Y)^*.
  \end{equation}
  These two 1-morphisms are (up to invertible 2-cells) inverse to each
  other. This shows the first claim. The existence and the
  pseudo-naturality of the 2-isomorphism $\xi_{f,g}$ now follows from
  the definition of $\xi$ and
  lemma \ref{lem:ev-dinatural}.
\end{proof}
Now, we show that the evaluation 1-morphism respects the monoidal
structure:
\begin{lemma}
  \label{lem:ev-monoidal}
 For a dualizable object $X$ of a
  symmetric monoidal bicategory
 $\cC$, the evaluation 1-morphism
  $\ev_X$ is a \emph{monoidal} pseudo-dinatural transformation: namely, the
  following diagram commutes up to 2-isomorphism.
 \begin{equation}
   \label{eq:lem-ev-1}
   \vcenter{\hbox{\includegraphics{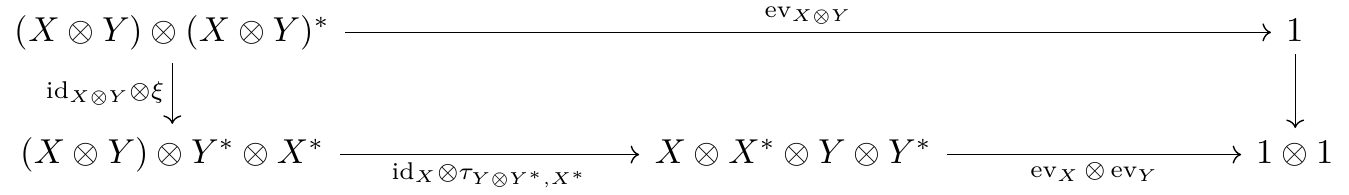}}}
 \end{equation}
  Here, the 1-equivalence $\xi$ is due to Lemma
  \ref{lem:tensor-product-dual}.
\end{lemma}
\begin{proof}
Let us construct a 2-isomorphism in the diagram in Equation \eqref{eq:lem-ev-1}.
  Consider the diagram in figure \ref{fig:ev-monoidal} on page
  \pageref{fig:ev-monoidal}: here, the composition of the horizontal
  arrows at the top, together with the two arrows on the vertical
  right are exactly the 1-morphism in Equation
  \eqref{eq:lem-ev-1}. The other arrow is given by $\ev_{X \ot Y}$. We
  have not written down the tensor product, and left out isomorphisms
  of the form $1 \ot X \cong X \cong X \ot 1$ for readability.
\end{proof}
We can now establish the monoidality of the right adjoint of the
evaluation via the following lemma.
\begin{lemma}
  \label{lem:right-adjoint-ev-monoidal}
  Let $\cC$ a symmetric monoidal bicategory, and let $X$ and $Y$ be
  fully-dualizable objects. Then, the right adjoint of the evaluation
  is monoidal. More precisely: if $\xi: (X \ot Y)^* \to Y^* \ot X^*$
  is the 1-equivalence of Lemma \ref{lem:tensor-product-dual}, the
  following diagram commutes up to 2-isomorphism.
  \begin{equation}
    \vcenter{\hbox{\includegraphics{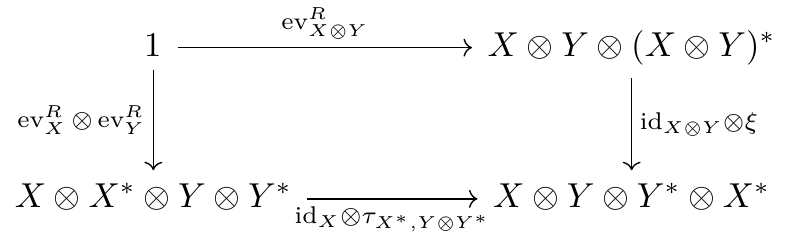}}}
  \end{equation}
\end{lemma}
\begin{proof}
  In a first step, we show that the right adjoint of the 1-morphism
  \begin{equation}
    \label{eq:evr-r-1}
    (\ev_X \ot \ev_Y) \circ (\id_X \ot \tau_{Y \ot Y^*, X^*}) \circ
    (\id_{X \ot Y}\ot \xi)
  \end{equation}
  is given by the 1-morphism
  \begin{equation}
    \label{eq:evr-r-2}
    (\id_{X \ot Y} \ot \xi^{-1}) \circ (\id_X \circ \tau_{X^*, Y \ot
      Y^*}) \circ (\ev^R_X \ot \ev_Y^R).
  \end{equation}
  Indeed, if
  \begin{equation}
    \begin{aligned}
      \eta_X: \id_{X \ot X^*} & \to \ev_X^R \circ \ev_X \\
      \eps_X: \ev_X \circ \ev^R_X & \to \id_1
    \end{aligned}
  \end{equation}
  are the unit and counit of the right-adjunction of $\ev_X$ and its
  right adjoint $\ev_X^R$, we construct adjunction data for the
  adjunction in equations \eqref{eq:evr-r-1} and \eqref{eq:evr-r-2} as
  follows. Let $\tilde \eps$ and $\tilde \eta$ be the following
  2-morphisms:
  \begin{align}
    \begin{aligned}
      &\tilde \eps: \! \begin{multlined}[t] ( \ev_X \ot \ev_Y) \circ
        (\id_X \ot \tau_{Y \ot Y^*, X^*}) \circ (\id_{X \ot Y} \ot
        \xi) \circ (\id_{X \ot Y} \ot
        \xi^{-1}) \\
        \circ (\id_X \ot \tau_{X^*, Y \ot Y^*}) \circ (\ev_X^R \ot
        \ev_Y^R)
      \end{multlined}\\
      & \cong ( \ev_X \ot \ev_Y) \circ (\id_X \ot \tau_{Y \ot Y^*,
        X^*}) \circ (\id_X \ot \tau_{X^*, Y \ot Y^*}) \circ (\ev_X^R
      \ot \ev_Y^R)  \\
      & \cong ( \ev_X \ot \ev_Y) \circ (\ev_X^R \ot \ev_Y^R)
      \xrightarrow{\eps_X \ot \eps_Y} \id_1
    \end{aligned}
  \end{align}
  and
  \begin{align}
    \begin{aligned}
      \tilde \eta : \id_{X \ot Y \ot (X \ot Y)^*} &\cong (\id_{X \ot
        Y} \ot
      \xi^{-1}) \circ  (\id_{X \ot Y} \ot \xi) \\
      & \cong (\id_{X \ot Y} \ot \xi^{-1}) \circ (\id_X \ot \tau_{X^*,
        Y \ot Y^*}) \circ (\id_X \ot \tau_{Y \ot Y^*,
        X^*}) \circ (\id_{X \ot Y} \ot \xi)  \\
      & \xrightarrow{\id \ot \eta_X \ot \eta_Y \ot \id}
      \! \begin{multlined}[t] (\id_{X \ot Y} \ot \xi^{-1}) \circ
        (\id_X \ot \tau_{X^*, Y \ot Y^*}) \circ (\ev_X^R \ot \ev_Y^R)
        \\ \circ ( \ev_X \ot \ev_Y) \circ (\id_X \ot \tau_{Y \ot Y^*,
          X^*}) \circ (\id_{X \ot Y} \ot \xi)
      \end{multlined}
    \end{aligned}
  \end{align}
  One now shows that the two 1-morphisms in Equation
  \eqref{eq:evr-r-1} and \eqref{eq:evr-r-2}, together with the two
  2-morphisms $\tilde \eps$ and $\tilde \eta$ form an adjunction. This
  gives that the two 1-morphisms in Equations \eqref{eq:evr-r-1} and
  \eqref{eq:evr-r-2} are adjoint.

  Next, notice that the 1-morphism in Equation \eqref{eq:evr-r-1} is
  isomorphic to the 1-morphism $\ev_{X \ot Y}$ by Lemma
  \ref{lem:ev-monoidal}. Thus, the right adjoint of $\ev_{X \ot Y}$ is
  given by the right adjoint of the 1-morphism in Equation
  \eqref{eq:evr-r-1}, which is the 1-morphism in Equation
  \eqref{eq:evr-r-2} by the argument above. Since all adjoints are
  equivalent, this shows the lemma.
\end{proof}
We are now ready to prove that the Serre automorphism is a
\emph{monoidal} pseudo-natural transformation.
\begin{prop}
  Let $\cC$ be a symmetric monoidal bicategory. Then, the Serre
  automorphism is a \emph{monoidal} pseudo-natural transformation of
  $\Id_{\core{\cC^\fd}}$.
\end{prop}
\begin{proof}
  By definition (cf. \cite[Definition
  2.7]{schommerpries-classification}), we have to provide invertible
  2-cells
  \begin{equation}
    \begin{aligned}
      \Pi_{X,Y} &: S_{X \ot Y} \to S_X \ot S_Y \\
      M & :  S_1 \to \id_1,
    \end{aligned}
  \end{equation}
  satisfying suitable coherence equations.
  By the definition of the Serre automorphism in Definition
  \ref{def:serre-auto}, it suffices to show that the evaluation and
  its right adjoint are monoidal, since the braiding $\tau$ will be
  monoidal by definition. The monoidality of the evaluation is proven
  in Lemma \ref{lem:ev-monoidal}, while the monoidality of its right
  adjoint follows from Lemma
  \ref{lem:right-adjoint-ev-monoidal}. These two lemmas thus provide
  an invertible 2-cell $S_{X \ot Y} \cong S_X \ot S_Y$. The second
  2-cell $\id_1 \to S_1$ can be constructed in a similar way, by
  noticing that $1 \cong 1^*$.

  The three coherence equations for a pseudo-natural transformation
  now read
  \begin{equation}
    \begin{aligned}
      \Pi_{X \ot Y, Z} \circ (\Pi_{X \ot Y} \ot \id_{S_Z}) &= \Pi_{X, Y
        \ot Z} \circ (\id_{S_X} \ot \Pi_{Y,Z}) \\
      \Pi_{1,X} &= M \ot \id_{S_X} \\
      \Pi_{X,1} & = \id_{S_X} \ot M
    \end{aligned}
  \end{equation}
  and can be checked directly by hand.
\end{proof}

\section{Monoidal homotopy actions}
\label{sec:triviality-actions}
In this section, we investigate homotopy actions on symmetric monoidal
bicategories. In particular, we are interested in the case when the
group action is compatible with the monoidal structure.  By a
(homotopy) action of a topological group $G$ on a bicategory $\cC$, we
mean a weak monoidal 2-functor $\rho: \Pi_2(G) \to \Aut(\cC)$, where
$\Pi_2(G)$ is the fundamental 2-groupoid of $G$, and $\Aut(\cC)$ is
the bicategory of auto-equivalences of $\cC$.  For details on homotopy
actions of groups on bicategories, we
refer the reader to \cite{hsv16}. \\
In order to simplify the exposition, we introduce the following
\begin{newdef}
  Let $G$ be a topological group. We will say that $G$ is 2-truncated
  if $\pi_2(G,x)$ is trivial for every base point $x \in G$.
\end{newdef}
Moreover, we will need also the following definition.
\begin{newdef}
  Let $\cC$ be a symmetric monoidal bicategory. We will say that $\cC$
  is algebraically 1-connected if it is monoidally equivalent to $B^{2}H$, for some abelian group $H$.
\end{newdef}
In the following, we denote by $\Aut_\ot(\cC)$ the monoidal bicategory of
invertible monoidal weak 2-functors of $\cC$, invertible monoidal
pseudo-natural transformations, and invertible monoidal
modifications. Details of the construction can be found in
\cite[Appendix A]{hesse-thesis}.
\begin{newdef}
  Let $\cC$ be a symmetric monoidal category and $G$ be a topological
  group. A \emph{monoidal} homotopy action of $G$ on $\cC$ is a
  monoidal morphism $\rho: \Pi_2(G) \to \Aut_\ot(\cC)$.
\end{newdef}

We now prove a general criterion for when monoidal homotopy actions
are trivializable.
\begin{prop}
  \label{prop:action-trivializable}
  Let $\cC$ be a symmetric monoidal bicategory, and let $G$ be a path
  connected topological group. Assume that $G$ is 2-truncated, and that  $\Aut_\ot(\cC)$ is algebraically 1-connected, with abelian group $H$. If
  $H^{2}_{grp}(\pi_{1}(G,e), H)\simeq 0$, then any monoidal homotopy
  action of $G$ on $\cC$ is pseudo-naturally isomorphic to the trivial
  action.
\end{prop}
\begin{proof}
  Let $\rho: \Pi_2(G) \to \Aut_\ot(\cC)$ be a weak monoidal 2-functor.
  Since $\Aut_\ot(\cC)$ was assumed to be monoidally equivalent to
  $B^2H$ for some abelian group $H$, the group action $\rho$ is
  equivalent to a weak monoidal 2-functor $\rho: \Pi_2(G) \to
  B^2H$. Due to the fact that $G$ is path connected and 2-truncated,
  we have that $\Pi_2(G)\simeq B\pi_{1}(G, e)$, where $\pi_{1}(G, e)$
  is regarded as a discrete monoidal category. Thus, the monoidal
  homotopy action $\rho$ is monoidally equivalent to a weak monoidal
  2-functor $B\pi_{1}(G, e)\to{B^{2}H}$.

  We claim that such functors are classified by
  $H^2_{grp}(\pi_1(G,e), H)$ up to monoidal pseudo-natural isomorphism. Indeed,
  let $F:B\pi_1(G,e) \to B^2 H$ be a weak monoidal 2-functor. It is
  easy to see that $F$ is trivial as a weak 2-functor, since we must
  have $F(*)=*$ on objects, $F(\gamma) = \id_*$ on 1-morphisms, and
  $B \pi_1(G)$ only has identity 2-morphisms. Thus, the only
  non-trivial data of $F$ can come from the monoidal structure on $F$.
  The 1-dimensional components of the pseudo-natural transformations
  $\chi_{a,b}: F(a) \ot F(b) \to F(a \ot b)$ must be trivial since
  there are only identity 1-morphisms in $B^2H$. The 2-dimensional
  components of this pseudo-natural transformation consists of a
  2-morphism $\chi_{\gamma, \gamma'}$ in $B^2H$ for every pair of
  1-morphisms $\gamma:a \to b$ and $\gamma': a' \to b'$ in
  $B \pi_1(G)$ in the diagram in equation
  \eqref{eq:non-trivial-cocycle} below.
  \begin{equation}
    \label{eq:non-trivial-cocycle}
      \vcenter{\hbox{\includegraphics{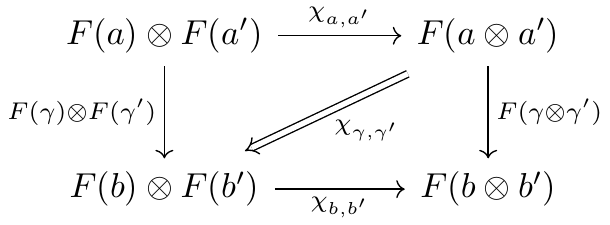}}}
   \end{equation}
Hence, we obtain a 2-cochain $\pi_1(G) \times \pi_1(G) \to H$, which
obeys the cocycle condition due to the coherence equations of a
monoidal 2-functor, cf. \cite[Definition 2.5]{schommerpries-classification}.

One now checks that a monoidal pseudo-natural transformation between
two such functors is exactly a 2-coboundary, which shows the
claim. Since we assumed that $H^2_{grp}(\pi_1(G,e), H) \simeq 0$, the
original action $\rho$ must be trivializable.
\end{proof}
Next, we show that the bicategory $\Alg_2^\fd$ of finite-dimensional,
semi-simple algebras, bimodules and intertwiners, equipped with the
monoidal structure given by the \emph{direct sum} fulfills the
conditions of Proposition \ref{prop:action-trivializable}.
\begin{lemma}
\label{lem:algdsum}
Let $\K$ be an algebraically closed field. Let $\cC=\Alg_2^\fd$ be the
bicategory where objects are given by finite-dimensional, semi-simple
algebras, equipped with the monoidal structure given by the direct
sum. By viewing $\cC$ with the monoidal structure equipped by
  the direct sum, $\cC$ turns into a \emph{linear} bicategory. Then, $\Aut_\ot(\cC)$ and $B^2\K^*$ are equivalent as
  symmetric monoidal bicategories.
\end{lemma}
\begin{proof}
  Let $F: \Alg_2^\fd \to \Alg_2^\fd$ be a weak monoidal 2-equivalence,
  and let $A$ be a finite-dimensional, semi-simple algebra. Then $A$
  is isomorphic to a direct sum of matrix algebras. Calculating up to
  Morita equivalence and using that $F$ has to preserve the single
  simple object $\K$ of $\Alg_2$, we have
  \begin{equation}
    F(A)  \cong F \left(\bigoplus_{i} M_{n_i}(\K) \right) \cong
    \bigoplus_{i} F \left( M_{n_i} (\K) \right)  \cong \bigoplus_i F
    (\K)  \cong \bigoplus_i \K  \cong \bigoplus_{i} M_{n_i} (\K) \cong A.
  \end{equation}
  A straightforward calculation using basic linear algebra
    confirms that these isomorphisms are even pseudo-natural.
  Thus, the functor $F$ is pseudo-naturally isomorphic to the identity
  functor on $\Alg_2^\fd$.

  Now, let $\eta: F \to G$ be a monoidal pseudo-natural isomorphism between two
  endofunctors of $\Alg_2$. Since both $F$ and $G$ are
  pseudo-naturally isomorphic to the identity, we may consider instead
  a pseudo-natural isomorphism $\eta: \id_{\Alg_2^\fd} \to
  \id_{\Alg_2^\fd}$. We claim that up to an invertible modification,
  the 1-equivalence $\eta_A : A \to A$ must be given by the bimodule
  $_AA_A$, which is the identity 1-morphism on $A$ in
  $\Alg_2$. Indeed, since $\eta_A$ is assumed to be linear, it
  suffices to consider the case of $A=M_n(\K)$ and to take direct
  sums. It is well-known that the only simple modules of $A$ are given
  by $\K^n$. Thus,
  \begin{equation}
    \eta_A = (\K^n)^\alpha \ot_\K (\K^n)^\beta,
  \end{equation}
  where $\alpha$ and $\beta$ are multiplicities. Now, \cite[Lemma
  2.6]{hsv16} ensures that these multiplicities are trivial, and thus
  we have $\eta_A = {_AA_A}$ up to an invertible intertwiner. This
  shows that up to invertible modifications, all 1-morphisms in
  $\Aut_\ot(\Alg_2^\fd)$ must be identities.

  Now, let $m$ be a monoidal invertible endo-modification of the
  pseudo-natural transformation $\id_{\id_{\Alg_2^\fd}}$. Then, the
  component $m_A : {_AA_A} \to {_AA_A} $ is an element of
  $\End_{(A,A)}(A) \cong \K$. As the modification square commutes
  automatically, this show that the 2-morphisms of
  $\Aut_\ot(\Alg_2^\fd)$ stand in bijection to $\K^*$.
\end{proof}

\begin{remark}
  Notice that the symmetric monoidal structure on $\Alg_2^\fd$
  considered above is \emph{not} the standard one, which is instead
  the one induced by the tensor product of algebras, and which is the
  monoidal structure relevant for the remainder of the paper.
\end{remark}
The last lemmas imply the following
\begin{lemma}
  \label{cor:action-alg2-trivializable}
  Any monoidal $SO(2)$-action on $\Alg_2^\fd$ equipped with the
  monoidal structure given by the direct sum is trivial.
  \end{lemma}
  \begin{proof}
    Since $\pi_{1}(SO(2), e)\simeq\mathbb{Z}$, and
    $H^{2}_{grp}(\mathbb{Z}, \mathbb{K}^{*})\simeq H^{2}(S^{1},
    \mathbb{K}^{*})\simeq 0$, Proposition
    \ref{prop:action-trivializable} and Lemma \ref{lem:algdsum} ensure
    that any monoidal $SO(2)$-action on $\Alg_{2}^{\fd}$ is
    trivializable.
  \end{proof}
  Recall that we regarded $\cC=\Alg_2^{\fd}$ as a monoidal
  bicategory with the monoidal structure given by direct sums.
\begin{cor}
  Since $\Alg_2^\fd$ and $\Vect_2^\fd$ are equivalent as additive
  categories, any $SO(2)$-action on $\Vect_2^\fd$ via linear morphisms
  is trivializable.
\end{cor}
\begin{remark}
  The last two results rely on the fact that
  $\Aut_\ot(\Alg_2^\fd)$ and $\Aut_\ot(\Vect^\fd_2)$ are 1-connected
  as additive categories. This is due to the fact that
  fully-dualizable part of either $\Alg_2$ or $\Vect_2$ is
  semi-simple. An example in which the conditions in Proposition
  \ref{prop:action-trivializable} do \emph{not} hold is provided by
  the bicategory of Landau-Ginzburg models.
\end{remark}

\section{Computing homotopy fixed points}
\label{sec:comp-homot-fixed}
In this Section, we explicitly compute the bicategory of homotopy
fixed points of an $SO(2)$-action which is induced by an arbitrary
pseudo-natural equivalence of the identity functor of an arbitrary
bicategory $\cC$. Recall that a $G$-action on a bicategory $\cC$ is a
monoidal 2-functor $\rho: \Pi_2(G) \to \Aut(\cC)$, or equivalently a
trifunctor $\rho: B\Pi_2(G) \to \Bicat$ with $\rho(*)=\cC$. The
bicategory of homotopy fixed points $\cC^G$ is then given by the
tri-limit of this trifunctor.

In $\Bicat$, the tricategory of bicategory, this trilimit can be
computed as follows: if $\Delta : B\Pi_2(G) \to \Bicat$ is the
constant functor assigning to the one object $*$ the terminal
bicategory with one object, the trilimit of the action functor $\rho$
is given by
\begin{equation}
  \cC^G:=\lim \rho = \Nat(\Delta, \rho),  
\end{equation}
the bicategory of tri-transformations between $\rho$ and
$\Delta$. This definition is explicitly spelled out in \cite[Remark
3.11]{hsv16}. We begin by defining an $SO(2)$-action on an arbitrary
symmetric monoidal bicategory, starting from a pseudo-natural
transformation of the identity functor on $\cC$.

\begin{newdef}
  \label{def:non-triv-so2-action}
  Since $\Pi_2(SO(2))$ is equivalent to the bicategory with one
  object, $\Z$ worth of morphisms, and only identity 2-morphisms, we
  may define an $SO(2)$-action $\rho: \Pi_2(SO(2)) \to \Aut_\ot(\cC)$
  by the following data:
  \begin{itemize}
  \item For every group element $g \in SO(2)$, we assign the identity
    functor of $\cC$.
  \item For the generator $1 \in \Z$, we assign the pseudo-natural
    transformation of the identity functor given by $\alpha$. Due to
    the monoidality, this determines the value of $\rho$ on an
    arbitrary integer.
  \item Since there are only identity 2-morphisms in $\Z$, we have to
    assign these to identity 2-morphisms in $\cC$.
  \item For composition of 1-morphisms, we assign the invertible
    modification $\rho(a+b) \cong \rho(a)\circ \rho(b)$ coming from
    the fact that $\alpha$ is a \emph{monoidal} pseudo-natural
    transformation with respect to composition, which is the monoidal
    product in $\Aut_\ot(\cC)$.
  \item In order to make $\rho$ into a monoidal 2-functor, we have to
    assign additional data which we can choose to be trivial. In
    detail, we set $\rho(g \ot h):= \rho(g) \ot \rho(h)$, and
    $\rho(e):= \id_{\cC}$. Finally, we choose $\omega$, $\gamma$ and
    $\delta$ as in \cite[Remark 3.8]{hsv16} to be identities.
  \end{itemize}
\end{newdef}
For a proof that this defines indeed a weak 2-functor, we refer to
\cite[Lemma 3.2.3]{davi11}.

Our main example is the action of the Serre automorphism on the core
of fully-dualizable objects:
\begin{ex}
  \label{ex:serre-action-core-c-fd}
  If $\cC$ is a symmetric monoidal bicategory, consider
  $\core{\cC^\fd}$, the core of the fully-dualizable objects of
  $\cC$. By Proposition \ref{prop:serre-pseudo-natural-core-fd}, the
  Serre automorphism defines a pseudo-natural equivalence of the
  identity functor on $\core{\cC^\fd}$. By Definition
  \ref{def:non-triv-so2-action}, we obtain an $SO(2)$-action on
  $\core{\cC^\fd}$, which we denote by $\rho^{S}$.
\end{ex}
The next theorem computes the bicategory of homotopy fixed points
$\cC^{SO(2)}$ of the action in Definition
\ref{def:non-triv-so2-action}. This theorem generalizes \cite[Theorem
4.1]{hsv16}, which only computes the bicategory of homotopy fixed
points of the \emph{trivial} $SO(2)$-action.
\begin{theorem}
  \label{thm:fixed-points-non-triv-so2-action}
  Let $\cC$ be a symmetric monoidal bicategory, and let $\alpha:
  \id_\cC \to \id_\cC$ be a monoidal pseudo-natural equivalence of the identity
  functor on $\cC$. Let $\rho$ be the $SO(2)$-action on $\cC$ as in
  Definition \ref{def:non-triv-so2-action}.  Then, the bicategory of
  homotopy fixed points $\cC^G$ is equivalent to the bicategory with
  \begin{itemize}
  \item objects: $(c,\lambda)$ where $c$ is an object of $\cC$ and
    $\lambda : \alpha_c \to \id_c$ is a 2-isomorphism,
  \item 1-morphisms $(c, \lambda) \to (c',\lambda')$ in $\cC^G$ are
    given by 1-morphisms $f: c \to c'$ in $\cC$, so that the diagram
    \begin{equation}
      \label{eq:better-thm-condition-1mor}
      \vcenter{\hbox{\includegraphics{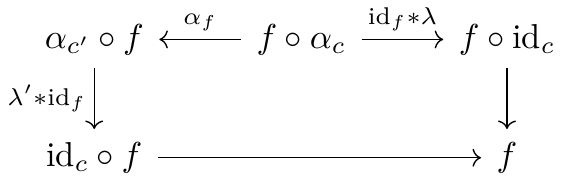}}}
    \end{equation}
    commutes,
  \item 2-morphisms of $\cC^G$ are given by 2-morphisms in $\cC$.
  \end{itemize}
\end{theorem}
\begin{proof}
  In order to prove the theorem, we need to explicitly unpack the
  definition of the bicategory of homotopy fixed points $\cC^G$. This
  is done in \cite[Remark 3.11 - 3.14]{hsv16}. In the following, we
  will use the notation introduced in \cite{hsv16}.

  The idea of the proof is to show that the forgetful functor which on
  objects of $\cC^G$ forgets the data $\Theta$, $\Pi$ and $M$ is an
  equivalence of bicategories. In order to show this, we need to
  analyze the bicategory of homotopy fixed points. We start with the
  objects of $\cC^G$.

  By definition, a homotopy fixed point of this action consists of
  \begin{itemize}
  \item An object $c \in \cC$,
  \item A 1-equivalence $\Theta: c \to c$,
  \item For every $n \in \Z$, an invertible 2-morphism $\Theta_n :
    \alpha_c^n \circ \Theta \to \Theta \circ \id_c$ so that $(\Theta,
    \Theta_n)$ fulfill the axioms of a pseudo-natural transformation,
  \item A 2-isomorphism $\Pi: \Theta \circ \Theta \to \Theta$ which
    obeys the modification square,
  \item Another 2-isomorphism $M: \Theta \to \id_c$
  \end{itemize}
  so that the following equations hold: Equation 3.18 of \cite{hsv16}
  demands that
  \begin{equation}
    \label{eq:fixed-point-condition1-1}
    \Pi \circ (\id_\Theta * \Pi)  = \Pi \circ (\Pi * \id_\Theta)
  \end{equation}
  whereas Equation 3.19 of \cite{hsv16}
  demands that $\Pi$ equals the composition
  \begin{equation}
    \label{eq:fixed-point-condition2-1}
    \Theta \circ \Theta \xrightarrow{ \id_\Theta * M} \Theta \circ \id_c \cong \Theta
  \end{equation} 
  and finally Equation 3.20 of \cite{hsv16}
  tells us that $\Pi$ must also be equal to the composition
  \begin{equation}
    \label{eq:fixed-point-condition3-1}
    \Theta \circ \Theta \xrightarrow{ M *  \id_\Theta }\id_c \circ \Theta \cong \Theta.
  \end{equation} 
  Hence $\Pi$ is fully specified by $M$. An explicit calculation using
  the two equations above then confirms that Equation
  \eqref{eq:fixed-point-condition1-1} is automatically
  fulfilled. Indeed, by composing with $\Pi^{-1}$ from the right, it
  suffices to show that $\id_\Theta * \Pi=\Pi* \id_\Theta$. Suppose
  for simplicity that $\cC$ is a strict 2-category. Then,
  \begin{equation}
    \begin{aligned}
      \id_\Theta * \Pi & = \id_\Theta *(M* \id_\Theta) && \qquad \text{by equation \eqref{eq:fixed-point-condition3-1}} \\
      &= (\id_\Theta*M)*\id_{\Theta} && \\
      &=\Pi* \id_\Theta && \qquad \text{by equation
        \eqref{eq:fixed-point-condition2-1}. }
    \end{aligned}
  \end{equation}
  Adding appropriate associators shows that this is true in a general
  bicategory.

  Note that by using the modification $M$, the 2-morphism
  $\Theta_n: \alpha_n^c \to \Theta \circ \id_c$ can be regarded as a
  2-morphism $\lambda_n : \alpha_c \to \id_c$. Here, $\alpha^n_c$ is
  the $n$-times composition of 1-morphism $\alpha_c$. Indeed, define
  $\lambda_n$ by setting
  \begin{equation}
    \label{eq:lambda}
    \lambda_n:=\left( \alpha_c \cong \alpha_c \circ \id_c \xrightarrow{\id_{\alpha_c} * M^{-1}} \alpha_c \circ \Theta
      \xrightarrow{\Theta_n} \Theta \circ \id_c \cong \Theta
      \xrightarrow{M} \id_c \right).
  \end{equation} 
  In a strict 2-category, the fact that $\Theta$ is a pseudo-natural
  transformation requires that $\lambda_0=\id_c$ and that
  $\lambda_n=\lambda_1 * \cdots * \lambda_1$. In a bicategory, similar
  equations hold by adding coherence morphisms. Thus,
  $\lambda_n$ is fully determined by $\lambda_1$. In order to simplify
  notation, we set $\lambda:=\lambda_1: \alpha_c \to \id_c$. 

  A 1-morphism of homotopy fixed points $(c, \Theta, \Theta_n, \Pi, M)
  \to (c', \Theta', \Theta_n', \Pi', M')$ consists of:
  \begin{itemize}
  \item a 1-morphism $f:c \to c'$,
  \item an invertible 2-morphism $m:f \circ \Theta \to \Theta' \circ
    f$ which fulfills the modification square. Note that $m$ is
    equivalent to a 2-isomorphism $m: f \to f'$ which can be seen by using the
    2-morphism $M$.
  \end{itemize}
  The condition due to Equation 3.24 of \cite{hsv16}
  demands that the following 2-isomorphism
  \begin{align}
    f \circ \Theta \xrightarrow{\id_f * M} f \circ \id_c \cong f
    \intertext{is equal to the 2-isomorphism} f \circ \Theta
    \xrightarrow{m} \Theta' \circ f \xrightarrow{M' * \id_f} \id_{c'}
    \circ f \cong f
  \end{align}
  and thus is equivalent to the equation
  \begin{equation}
    \label{eq:1-morphism-fixed-point-condition2-1}
    m=\left( f \circ \Theta \xrightarrow{\id_f * M} f \circ \id_c
      \cong f 
      \cong \id_{c'} \circ f \xrightarrow{{M'}^{-1}* \id_{f}} \Theta' \circ f \right)
  \end{equation}
  Thus, $m$ is fully determined by $M$ and $M'$.  The condition due to
  Equation 3.23 of \cite{hsv16}
  reads
  \begin{equation}
    \label{eq:1-morphism-fixed-point-condition1-1}
    m \circ (\id_f * \Pi) = (\Pi' * \id_f) \circ (\id_{\Theta'} * m) \circ (m * \id_\Theta)
  \end{equation}
  and is automatically satisfied, as an explicit calculation in
  \cite{hsv16} confirms.  Now, it suffices to look at the modification
  square of $m$, in Equation 3.25 of \cite{hsv16}. This condition is
  equivalent to the commutativity of the diagram
  \begin{equation}
    \vcenter{\hbox{\includegraphics{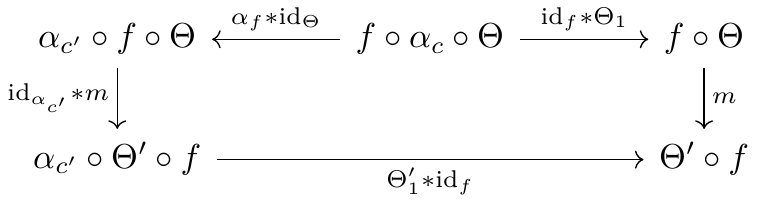}}}
  \end{equation}
  Substituting $m$ as in Equation
  \eqref{eq:1-morphism-fixed-point-condition2-1} and $\Theta_1$ for
  $\lambda:=\lambda_1$ as defined in Equation \eqref{eq:lambda}, one
  confirms that this diagram commutes if and only if the diagram in
  Equation \eqref{eq:better-thm-condition-1mor} commutes.

  If $(f,m)$ and $(g,n)$ are 1-morphisms of homotopy fixed points, a
  2-morphism of homotopy fixed points consists of a 2-isomorphism
  $\beta:f \to g$ in $\cC$. The condition coming from Equation
  3.26 of \cite{hsv16} then demands that the diagram
  \begin{equation}
    \label{eq:2-morphism-fixed-point-condition-1}
    \vcenter{\hbox{\includegraphics{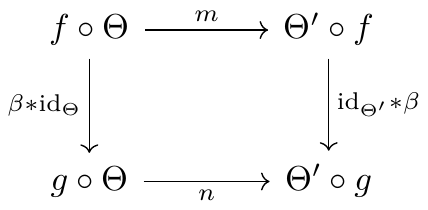}}}
  \end{equation}
  commutes. Using the fact that both $m$ and $n$ are uniquely
  specified by $M$ and $M'$, one quickly confirms that this diagram
  commutes automatically.

  Our detailed analysis of the bicategory $\cC^G$ shows that the
  forgetful functor which forgets the data $\Theta$, $M$, and $\Pi$ on
  objects and assigns $\Theta_1$ to $\lambda$, which forgets the data
  $m$ on 1-morphisms, and which is the identity on 2-morphisms is an
  equivalence of bicategories.
\end{proof}
\begin{cor}
  \label{cor:homotopy-fixed-points-serre}
  Let $\cC$ be a symmetric monoidal bicategory, and consider the
  $SO(2)$-action of the Serre automorphism on $\core{\cC^\fd}$ as in
  Example \ref{ex:serre-action-core-c-fd}. Then, the bicategory of
  homotopy fixed points $\core{\cC^\fd}^{SO(2)}$ is equivalent to a
  bicategory where
  \begin{itemize}
  \item objects are given by pairs $(X,\lambda_X)$ with $X$ a
    fully-dualizable object of $\cC$ and $\lambda_X : S_X \to \id_X$
    is a 2-isomorphism which trivializes the Serre automorphism,
  \item 1-morphisms are given by 1-equivalences $f:X \to Y$ in $\cC$,
    so that the diagram
    \begin{equation}
      \label{eq:serre-auto-fixed-point-condition}
      \vcenter{\hbox{\includegraphics{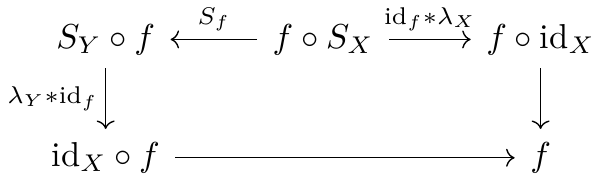}}}
    \end{equation}
    commutes, and
  \item 2-morphisms are given by 2-isomorphisms in $\cC$.
  \end{itemize}
\end{cor}
\begin{remark}
  Recall that we have defined the bicategory of homotopy fixed points
  $\cC^G$ as the tri-limit of the action considered as a trifunctor
  $\rho:B\Pi_2(G) \to \Bicat$. Since we only consider symmetric
  monoidal bicategories, we actually obtain an action with values in
  $\SymMonBicat$, the tricategory of symmetric monoidal
  bicategories. It would be interesting to compute the limit of the
  action in this tricategory. We expect that this trilimit computed in
  $\SymMonBicat$ is given by $\cC^G$ as a bicategory, with the
  symmetric monoidal structure induced by the symmetric monoidal
  structure of $\cC$.
\end{remark}
\begin{remark}
  By \cite{davi11}, the action via the Serre automorphism on
  $\core{\Alg_2^\fd}$ is trivializable. The category of homotopy fixed
  points $\core{\Alg_2^\fd}^{SO(2)}$ is then equivalent to the
  bigroupoid of symmetric, semi-simple Frobenius algebras.

  Similarly, the action of the Serre automorphism on $\Vect_2$ is
  trivializable. The bicategory of homotopy fixed points of this
  action is equivalent to the bicategory of finite Calabi-Yau
  categories, cf. \cite{hsv16}.
\end{remark}
\section{The 2-dimensional framed bordism bicategory}
\label{sec:fram-bord-bicat}
In this Section, we introduce a stricter version of the framed bordism
bicategory $\Cob_{2,1,0}^\fr$: this symmetric monoidal bicategory
$\F_{cfd}$ is the free bicategory of a coherent fully-dual pair as
introduced in \cite[Definition 3.13]{piotr14}.

In order to efficiently work with this symmetric monoidal bicategory
$\F_{cfd}$, we use a strictification result for symmetric monoidal
bicategories as proven in \cite[Proposition 13]{bart14}: any symmetric monoidal
bicategory is equivalent to a \emph{stringent} symmetric monoidal
  2-category, which can be completely described in terms of a wire diagram
calculus introduced in \cite{bart14}. In the following, we apply this
strictification result to the symmetric monoidal
bicategory $\F_{cfd}$, and provide a description using the wire diagram calculus developed in \cite{bart14}, which we also refer to for the definition of a stringent symmetric monoidal 2-category.

Using this description, we define a non-trivial $SO(2)$-action on
 $\F_{cfd}$. If $\cC$ is an arbitrary symmetric monoidal bicategory,
the action on $\F_{cfd}$ will induce an action on the functor
bicategory $\Fun_\ot(\F_{cfd}, \cC)$ of symmetric monoidal
functors. Using the Cobordism Hypothesis for framed manifolds, which
has been proven in the bicategorical framework in \cite{piotr14}, we
obtain an $SO(2)$-action on $\core{\cC^\fd}$. We show that this
induced action coming from the framed bordism bicategory is exactly
the action given by the Serre automorphism.

We begin by recasting the definition of $\F_{cfd}$ in terms of the wire diagram calculus.
\begin{newdef}
  \label{def:F-cfd}
  The symmetric monoidal bicategory $\F_{cfd}$ consists of
  \begin{itemize}
  \item 2 generating objects $L$ and $R$,
  \item 4 generating 1-morphisms, given by
    \begin{itemize}
    \item a 1-morphism $\coev: 1 \to R \ot L$, which we write as $
      \vcenter{\hbox{\includegraphics[scale=0.8]{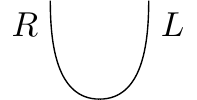}}}$ ,
    \item $\ev:L \ot R \to 1$ which we write as
      $\vcenter{\hbox{\includegraphics[scale=0.8]{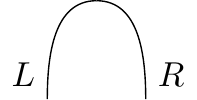}}}$,
    \item a 1-morphism $q: L \to L$,
    \item another 1-morphism $q^{-1}: L \to L$,
    \end{itemize}
  \item 12 generating 2-cells given by
    \begin{itemize}
    \item isomorphisms $\alpha$, $\beta$, $\alpha^{-1}$ and
      $\beta^{-1}$ as in Definition \ref{def:dual-pair}, which in
      pictorial form are given as follows:
      \begin{equation}
        \vcenter{\hbox{\includegraphics{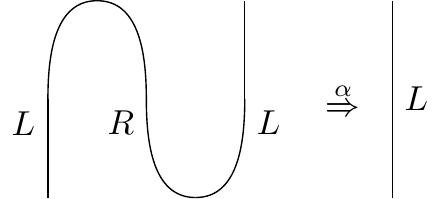}}}
        \hspace{2cm}
        \vcenter{\hbox{\includegraphics{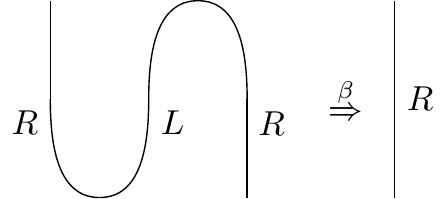}}}
      \end{equation}
    \item isomorphisms $\psi: q q^{-1} \cong \id_L : \psi^{-1}$ and
      $\phi:q^{-1} q \cong \id_L : \phi^{-1}$
    \item 2-cells $\mu_e : \id_1 \to \ev \circ \ev^L $ and $\eps_e:
      \ev^L \circ \ev \to \id_{L \ot R}$, where $\ev^L:= \tau \circ
      (\id_R \ot q^{-1}) \circ \coev$ which in pictorial form are
      given as follows:
      \begin{equation}
        \vcenter{\hbox{\includegraphics{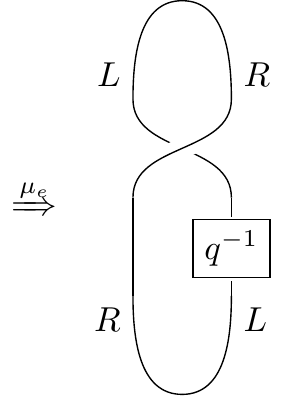}}}
        \hspace{2cm}
        \vcenter{\hbox{\includegraphics{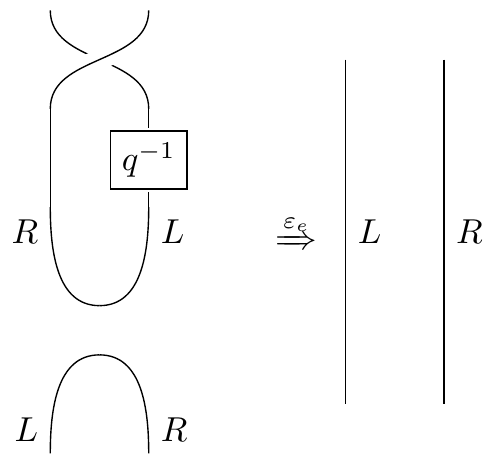}}}
      \end{equation}

    \item 2-cells $\mu_c : \id_{R \ot L} \to \coev \circ \coev^L$ and
      $\eps_c : \coev^L \circ \coev \to \id_1$, where ${\coev}^L:= \ev
      \circ (q \ot \id_R) \circ \tau$
    \end{itemize}
    which in pictorial form are given as follows:
    \begin{equation}
      \vcenter{\hbox{\includegraphics{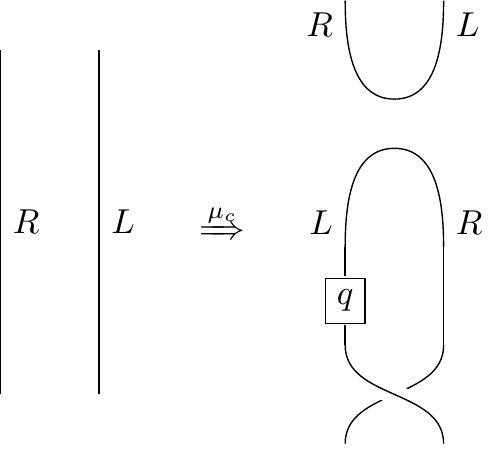}}}
      \hspace{2cm}
      \vcenter{\hbox{\includegraphics{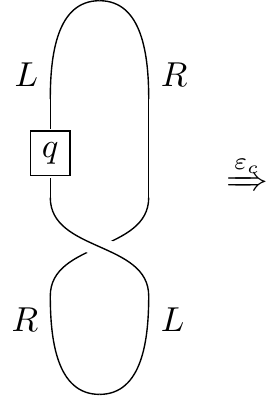}}}
    \end{equation}
  \end{itemize}
  so that the following relations hold:
  \begin{itemize}
  \item $\alpha$ and $\alpha^{-1}$, $\beta$ and $\beta^{-1}$, $\phi$
    and ${\phi}^{-1}$, $\psi$ and $\psi^{-1}$ are inverses to each
    other,
  \item $\mu_e$ and $\eps_e$ satisfy the two Zorro equations, which in
    pictorial form demand that the following composition of
    2-morphisms
    \begin{equation}
      \vcenter{\hbox{\includegraphics[scale=0.8]{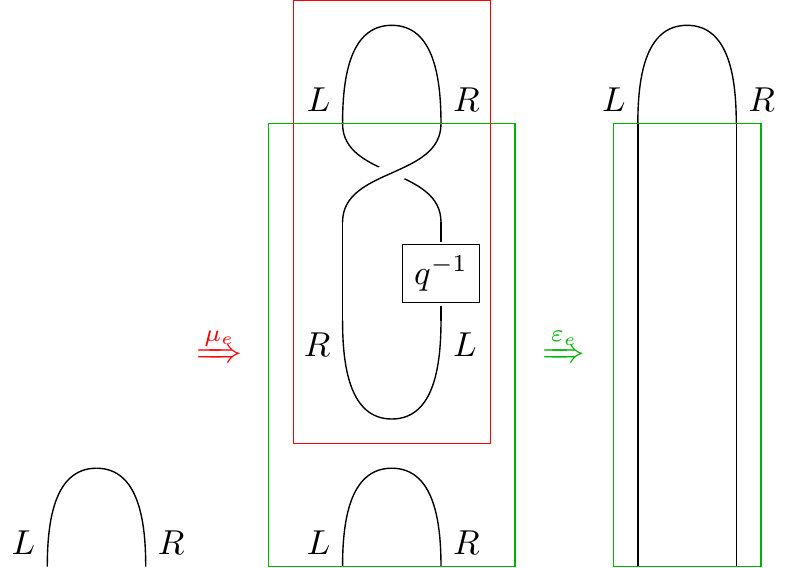}}}
    \end{equation}
    is equal to $\id_{\ev}$, and that the following composition of
    2-morphisms
    \begin{equation}
      \vcenter{\hbox{\includegraphics[scale=0.8]{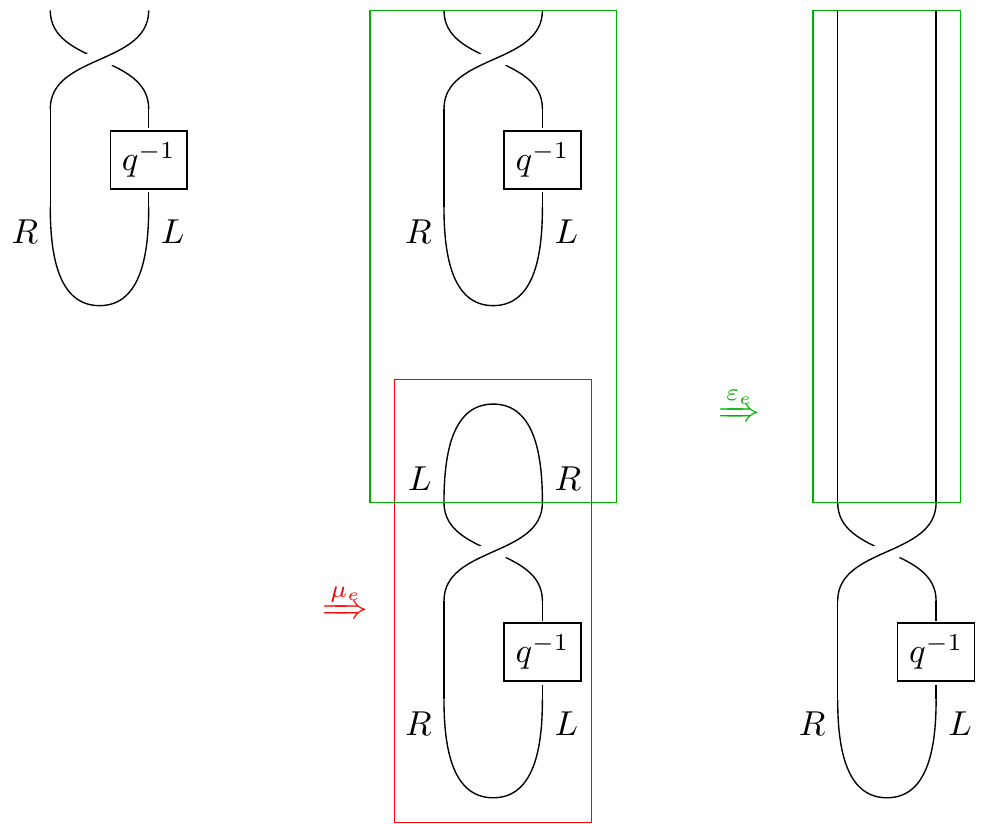}}}
    \end{equation}
    is equal to $\id_{\ev^L}$.
  \item $\mu_c$ and $\eps_c$ satisfy the two Zorro equations, which in
    pictorial form demand that the composition
    \begin{equation}
      \vcenter{\hbox{\includegraphics[scale=0.8]{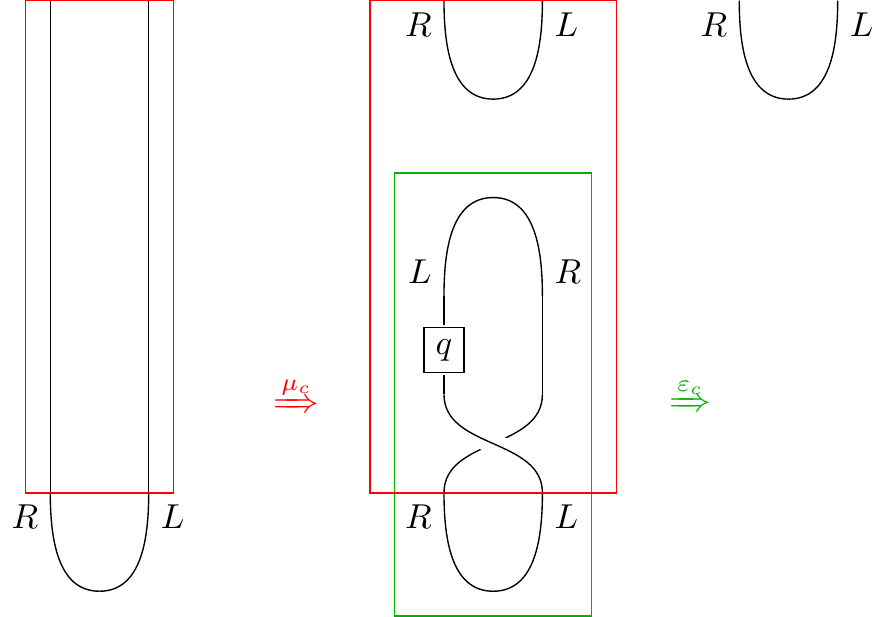}}}
    \end{equation}
    is equal to $\id_{\coev}$, and the composition of the following
    2-morphisms
    \begin{equation}
      \vcenter{\hbox{\includegraphics[scale=0.8]{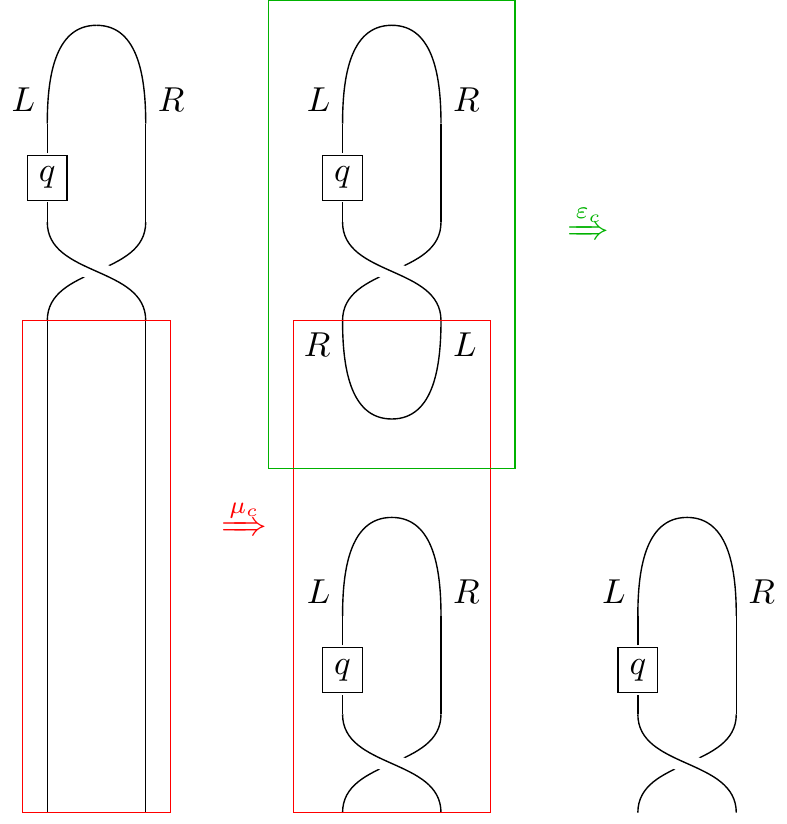}}}
    \end{equation}
    is equal to $\id_{\coev^L}$.
  \item $\phi$ and $\psi$ satisfy triangle identities,
  \item the cusp-counit equations in figure 5 and 6 on p.33 of
    \cite{piotr14} are satisfied,
  \item the swallowtail equations in figure 3 and 4 on p.15 of
    \cite{piotr14} are satisfied.
  \end{itemize}
\end{newdef}

\subsection{Action on the framed bordism bicategory}
We can now proceed to construct an $SO(2)$-action on $\F_{cfd}$. This action will
be vital for the remainder of the paper.

By Definition \ref{def:non-triv-so2-action} it suffices to construct a
pseudo-natural equivalence of the identity functor on $\F_{cfd}$ in
order to construct an $SO(2)$-action. This pseudo-natural
transformation is given as follows:
\begin{newdef}
  \label{def:action-framed-bordism-bicat}
  Let $\F_{cfd}$ be the free symmetric monoidal bicategory on a
  coherent fully-dual object as in Definition \ref{def:F-cfd}. We
  construct a pseudo-natural equivalence $\alpha : \id_{\F_{cfd}} \to
  \id_{\F_{cfd}}$ of the identity functor on $\F_{cfd}$ as follows:
  \begin{itemize}
  \item For every object $c$ of $\F_{cfd}$, we need to provide a
    1-equivalence $\alpha_c :c \to c$.
    \begin{itemize}
    \item For the object $L$ of $\F_{cfd}$, we define $\alpha_L:=q: L
      \to L$,
    \item for the object $R$ of $\F_{cfd}$, we set
      $\alpha_R:=(q^{-1})^*$, which in pictorial form is given by
      \begin{equation}
        \label{eq:q-inv-dual}
        (q^{-1})^*:=  \vcenter{\hbox{\includegraphics[]{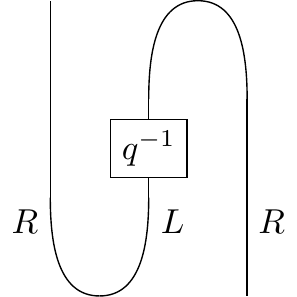}}}
      \end{equation}
    \end{itemize}
  \item for every 1-morphism $f:c \to d$ in $\F_{cfd}$, we need to
    provide a 2-isomorphism
    \begin{equation}
      \alpha_f:f \circ \alpha_c \to \alpha_d \circ f.
    \end{equation}
    \begin{itemize}
    \item For the 1-morphism $q:L \to L$ of $\F_{cfd}$ we define the
      2-isomorphism $\alpha_q:= \id_{q \circ q}$.
    \item For the 1-morphism ${q^{-1}}: L \to L$ we define the
      2-isomorphism
      \begin{equation}
        \alpha_{q^{-1}}:=\left( q^{-1} \circ q \xrightarrow{\phi} \id_L
          \xrightarrow{\psi^{-1}} q \circ q^{-1} \right).
      \end{equation}
    \item For the evaluation $\ev:L \ot R \to 1$, we define the
      2-isomorphism $\alpha_{\ev}$ to be the following composition:
    \end{itemize}
  \end{itemize}
  \begin{equation}
    \vcenter{\hbox{\includegraphics[width=\textwidth]{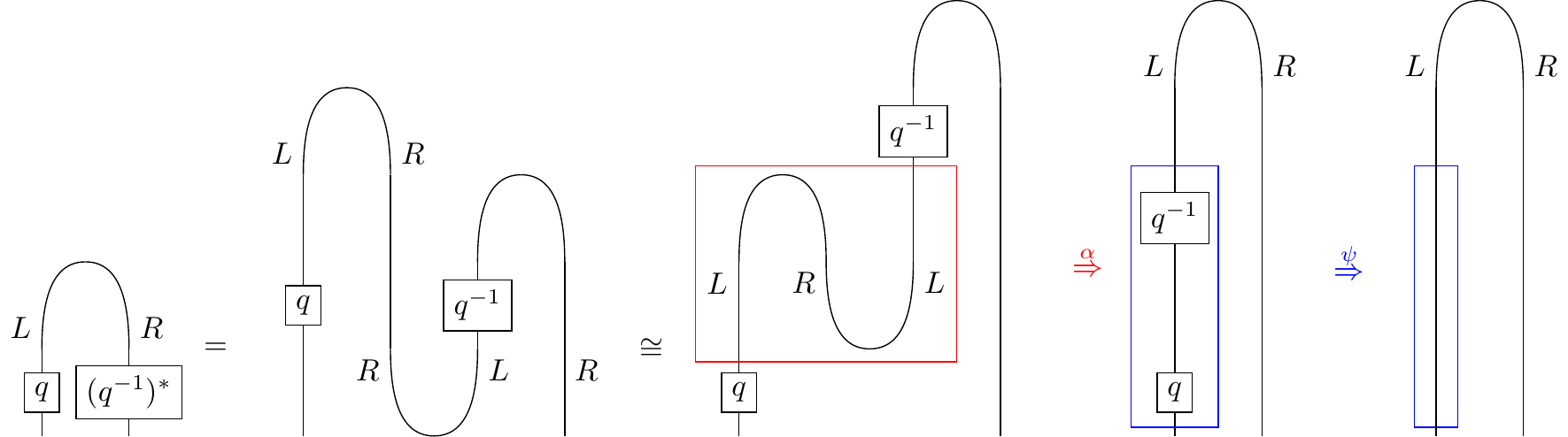}}}
  \end{equation}
  \begin{itemize}
  \item[]
    \begin{itemize}
    \item For the coevaluation $\coev:1 \to R \ot L$, we define the
      2-isomorphism $\alpha_{\coev}$ to be the composition
    \end{itemize}
  \end{itemize}
  \begin{equation}
    \vcenter{\hbox{\includegraphics[width=\textwidth]{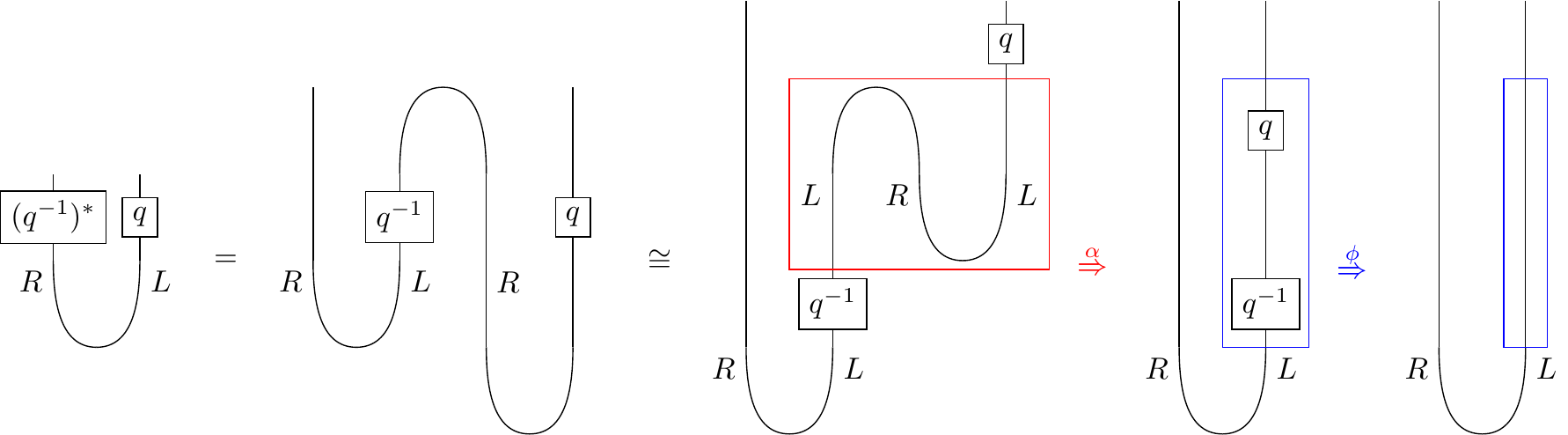}}}
  \end{equation}
\end{newdef}
One now checks that this defines a pseudo-natural transformation of
$\id_{\F_{cfd}}$. Using Definition \ref{def:non-triv-so2-action} gives
us a non-trivial $SO(2)$-action on $\F_{cfd}$.
\begin{remark}
  Note that the $SO(2)$-action on $\F_{cfd}$ does \emph{not} send
  generators to generators: for instance, the 1-morphism $(q^{-1})^*$
  in Equation \eqref{eq:q-inv-dual} is not part of the generating data
  of $\F_{cfd}$.
\end{remark}
\begin{remark}
  \label{action-fcd-symmetric-monoidal}
  Notice that the pseudo-natural equivalence
  $\alpha: \id_{\F_{cfd}} \to \id_{\F_{cfd}}$ constructed in
  Definition \ref{def:action-framed-bordism-bicat} is a
  \emph{monoidal} pseudo-natural transformation. This follows from the
  fact that we have defined $\alpha$ via generators and relations. In
  detail, we set
  \begin{equation}
    \begin{aligned}
      \alpha_X \ot \alpha_Y :&= \alpha_{X  \ot Y} \\
      \alpha_1 :&= \id_1.
    \end{aligned}
  \end{equation}
  Thus, we can choose the additional data $\Pi$ and $M$ of a monoidal
  pseudo-natural transformation to be trivial, and we obtain an
  $SO(2)$-action on $\F_{cfd}$ via symmetric monoidal morphisms.
\end{remark}

\subsection{Induced action on functor categories}
Starting from the action defined on $\F_{cfd}$, we induce an action on
the bicategory of functors $\Fun(\F_{cfd}, \cC)$ for an arbitrary
bicategory $\cC$. The construction of the induced action on the
bicategory of functors is a general construction. We provide details
in the following.
\begin{newdef}
  \label{def:induced-action-functor-bicats}
  Let $\rho: \Pi_2(G) \to \Aut(\cC)$ be a $G$-action on a bicategory
  $\cC$, and let $\cD$ be another bicategory. The $G$-action $\tilde
  \rho : \Pi_2(G) \to \Aut(\Fun(\cC,\cD))$ induced by $\rho$ is
  defined as follows:
  \begin{itemize}
  \item On objects $g \in G$, we define an endofunctor $\tilde
    \rho(g)$ of $\Fun(\cC,\cD)$ on objects $F$ on $\Fun(\cC,\cD)$ by
    $\tilde \rho (g)(F):= F \circ \rho(g^{-1})$. If $\alpha: F \to G$
    is a 1-morphism in $\Fun(\cC,\cD)$, we define
    \begin{equation}
      \tilde \rho(g)(\alpha):= \qquad
      \vcenter{\hbox{\includegraphics{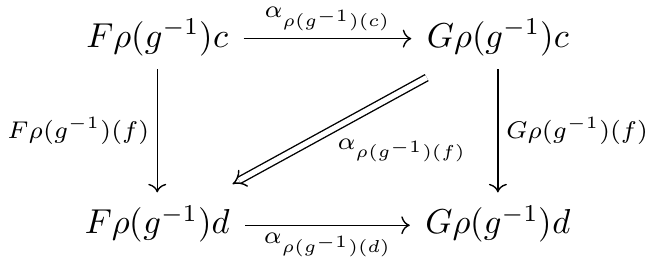}}}
    \end{equation}
    If $m:\alpha \to \beta$ is a 2-morphism in $\Fun(\cC,\cD)$, the
    value of $\tilde \rho(\gamma)$ is given by
    \begin{equation}
      \tilde \rho(\gamma)(m)_x:=m_{\rho(g^{-1})(x)}.
    \end{equation}
  \item on 1-morphisms $\gamma:g \to h$ of $\Pi_2(G)$, we define a
    1-morphism $ \tilde \rho(\gamma)$ in $\Aut(\Fun(\cC,\cD))$ between
    the two endofunctors $F \mapsto F \circ \rho(g^{-1})$ and $F
    \mapsto F \circ \rho(h^{-1})$ of $\Fun(\cC,\cD)$.

    Explicitly, this means:
    \begin{itemize}
    \item For each 2-functor $F: \cC \to \cD$, we need to provide a
      pseudo-natural transformation $\tilde \rho(\gamma)_F:F \circ
      \rho(g^{-1}) \to F \circ \rho(h^{-1})$ which we define via the
      diagram
      \begin{equation}
        \vcenter{\hbox{\includegraphics{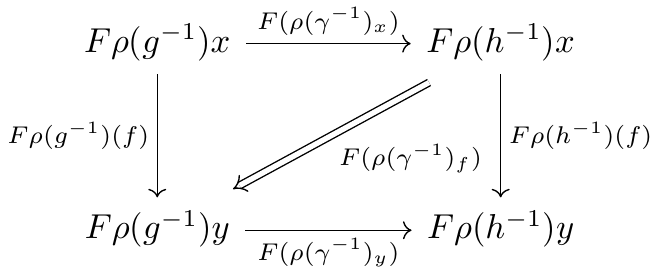}}}
      \end{equation}
      Here, $\gamma^{-1}$ is the \enquote{inverse} path of $\gamma$
      given by $t \mapsto \gamma(t)^{-1}$, and $f:x \to y$ is a
      1-morphism in $\cC$.
    \item For every pseudo-natural transformation $\alpha:F \to G$, we
      need to provide a modification $\tilde \rho(\gamma)_\alpha$ in
      the diagram
      \begin{equation}
        \vcenter{\hbox{\includegraphics{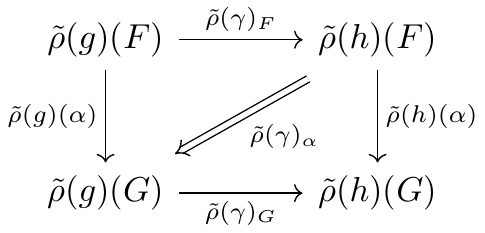}}}
      \end{equation}
      which we define by
      \begin{equation}
        \tilde \rho(\gamma)_\alpha := \alpha^{-1}_{\rho(\gamma^{-1})_x}.
      \end{equation}
    \end{itemize}
  \item For the 2-morphisms in $\Aut(\Fun(\cC,\cD))$ we proceed in a
    similar fashion: if $m: \gamma \to \gamma'$ is a 2-track, we have
    to provide a 2-morphism $\tilde \rho(m) : \tilde \rho(\gamma) \to
    \tilde \rho(\gamma')$ which can be done by explicitly writing down
    diagrams as above.
  \end{itemize}
  The rest of the data of a monoidal functor $ \tilde \rho$ is induced
  from the data of the monoidal functor $\rho$. 
\end{newdef}
For $\cC$ and $\mathcal{D}$ symmetric monoidal bicategories, the
bicategory of symmetric monoidal functors $\Fun_\ot(\cC,\cD)$ acquires
a monoidal structure by ``pointwise evaluation'' of functors. Such a
monoidal structure is also symmetric, see
\cite{schommerpries-classification}. The following result is
straightforward.
\begin{lemma}
\label{lem:induced-monoidal}
Let $\cC$ and $\mathcal{D}$ be symmetric monoidal bicategories, and
let $\mathcal{\rho}$ be a monoidal action of a group $G$ on
$\cC$. Then $\rho$ induces a monoidal action $\tilde \rho : \Pi_2(G)
\to \Aut_{\ot}(\Fun_\ot(\cC,\cD))$ .
\end{lemma} 

\begin{ex}
  \label{ex:induced-action-functor-bicat}
  Our main example for induced actions is the $SO(2)$-action on $\F_{cfd}$ as in
  Definition \ref{def:action-framed-bordism-bicat}. This action
  only depends on a pseudo-natural equivalence $\alpha$ of the
  identity functor on $\id_{\F_{cfd}}$. Consequently, the induced
  action on $\Fun(\F_{cfd}, \cC)$ also only depends on a
  pseudo-natural equivalence of the identity functor on
  $\Fun(\F_{cfd}, \cC)$. Using the definition above, we construct this
  induced pseudo-natural equivalence $\tilde \alpha$ as follows.
  \begin{itemize}
  \item For every 2-functor $F: \cC \to \cD$, we need to provide a
    pseudo-natural equivalence $\tilde \alpha_F:F \to F$, which is
    given by the diagram
    \begin{equation}
      \tilde \alpha_F := \qquad
      \vcenter{\hbox{\includegraphics{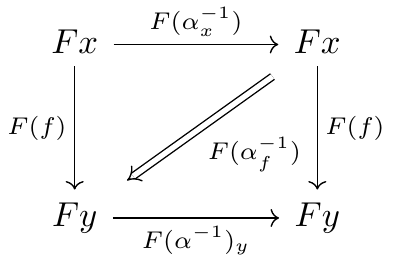}}}
    \end{equation}
  \item for every pseudo-natural transformation $\beta:F \to G$, we
    need to give a modification $\tilde \alpha_\beta$, which we define
    by the diagram
    \begin{equation}
      \vcenter{\hbox{\includegraphics{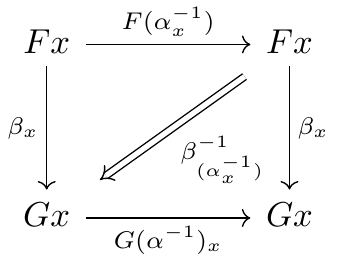}}}
    \end{equation}
  \end{itemize}
  This defines a pseudo-natural equivalence of the identity functor on
  $\Fun(\F_{cfd}, \cC)$. By Definition \ref{def:non-triv-so2-action},
  we obtain an $SO(2)$-action on $\Fun(\F_{cfd,} \cC)$. Note that
  $\F_{cfd}$ is even a \emph{symmetric monoidal} bicategory. The
  $SO(2)$-action on $\F_{cfd}$ of Definition
  \ref{def:action-framed-bordism-bicat} is via symmetric monoidal
  homomorphisms by Remark \ref{action-fcd-symmetric-monoidal}. Hence,
  if $\cC$ is also symmetric monoidal, then Lemma
  \ref{lem:induced-monoidal} provides a monoidal action on
  $\Fun_\otimes(\F_{cfd}, \cC)$.
\end{ex}

\subsection{Induced action on the core of fully-dualizable objects}
In this subsection, we compute the $SO(2)$-action on the core of
fully-dualizable objects coming from the $SO(2)$-action on
$\F_{cfd}$. Starting from the $SO(2)$-action on $\F_{cfd}$ as by
Definition \ref{def:action-framed-bordism-bicat}, we have shown in the
previous subsection how to induce an $SO(2)$-action on the bicategory
of symmetric monoidal functors $\Fun_\otimes(\F_{cfd}, \cC)$ for $\cC$
some symmetric monoidal bicategory. By the Cobordism Hypothesis for
framed manifolds, we obtain an induced $SO(2)$-action on
$\core{\cC^\fd}$. More precisely, denote by
\begin{equation}
    \label{eq:framed-cob-hyp-2}
    \begin{aligned}
      \ev_L : \Fun_\ot(\F_{cfd}, \cC) &\to \core{\cC^\fd} \\
      Z &\mapsto Z(L)
    \end{aligned}
  \end{equation}
  the evaluation map. The Cobordism Hypothesis for framed manifolds in
  two dimensions \cite{piotr14, Lurie09} states that $\ev_L$ is an
  equivalence of symmetric monoidal bicategories. Hence, the
  composition of the $SO(2)$-action on $\Fun_\ot(\F_{cfd} ,\cC)$ and
  (the inverse of) $\ev_L$ provides an $SO(2)$-action on
  $\core{\cC^\fd}$. The next proposition shows that this action is
  equivalent to the action $\rho^{S}$ induced by the Serre
  automorphism which is illustrated in Example
  \ref{ex:serre-action-core-c-fd}.
\begin{prop}
  \label{prop:serre-comes-from-fcfd}
  Let $\rho$ be the $SO(2)$-action on $\F_{cfd}$ given in Definition
  \ref{def:action-framed-bordism-bicat}, and let $\mathcal{C}$ be a
  symmetric monoidal bicategory. By Definition
  \ref{def:induced-action-functor-bicats}, we obtain a monoidal
  $SO(2)$-action on $\Fun_\ot(\F_{cfd}, \cC)$. Then, the monoidal
  $SO(2)$-action induced by the evaluation in Equation
  \eqref{eq:framed-cob-hyp-2} on $\core{\cC^\fd}$ is equivalent to
  $\rho^{S}$.
\end{prop}
\begin{proof}
  Let
  \begin{equation}
    \rho: \Pi_2(SO(2)) \to \Aut_{\otimes}(\Fun_\ot(\F_{cfd},\cC))  
  \end{equation}
  be the $SO(2)$-action on the bicategory of symmetric monoidal
  functors $\Fun_\ot(\F_{cfd}, \cC)$ as in Example
  \ref{ex:induced-action-functor-bicat}. This action
  only depends on a monoidal pseudo-natural transformation $\alpha$ on the
  identity functor on $\Fun_\ot(\F_{cfd}, \cC)$. By \cite{piotr14},
  the 2-functor in Equation \eqref{eq:framed-cob-hyp-2} which
  evaluates a framed field theory on the object $L$ is an equivalence
  of bicategories. Thus, we obtain an $SO(2)$-action $\rho'$ on
  $\core{\cC^\fd}$. This action is given as follows. By Definition
  \ref{def:non-triv-so2-action}, we only need to provide a
  monoidal pseudo-natural transformation of the identity functor of
  $\core{\cC^\fd}$. In order to write down this monoidal pseudo-natural
  transformation, note that the functor
  \begin{equation}
    \begin{aligned}
      \Aut_\otimes(\Fun_\ot(\F_{cfd},\cC)) &\to \Aut_\otimes(\core{\cC^\fd}) \\
      F & \mapsto \ev_L \circ F \circ \ev_L^{-1}
    \end{aligned}
  \end{equation}
  is a monoidal equivalence. Hence, the induced pseudo-natural transformation
  of $\id_{\core{\cC^\fd}}$ is given as follows:
  \begin{itemize}
  \item For each fully-dualizable object $c$ of $\cC$, we assign the
    1-morphism $\alpha'_c: c \to c$ defined by
    \begin{equation}
      \alpha_c':= \ev_L \left( \alpha_{\left( \ev_L^{-1}(c) \right)} \right)
    \end{equation}
  \item for each 1-equivalence $f:c \to d$ between fully-dualizable
    objects of $\cC$, we define a 2-isomorphism $\alpha'_f : f \circ
    \alpha'_c \to \alpha'_d \circ f$ by the formula
    \begin{equation}
      \alpha_f':= \ev_L \left( \alpha_{\left( \ev_L^{-1}(f) \right)} \right).
    \end{equation}
  \end{itemize}
  Here, $\alpha$ is the pseudo-natural transformation as in Example
  \ref{ex:induced-action-functor-bicat}. In order to see that
  $\alpha'_c$ is given by the Serre automorphism of the
  fully-dualizable object $c$, note that the 1-morphism $q:L \to L$ of
  $\F_{cfd}$ is mapped to the Serre automorphism $S_{Z(L)}$ by the
  equivalence in Equation \eqref{eq:framed-cob-hyp-2}.
\end{proof}

\begin{cor}
  \label{cor:equivalence-fixed-point-bicats}
  Let $\rho$ be the $SO(2)$-action on $\F_{cfd}$ given in Definition
  \ref{def:action-framed-bordism-bicat}, and let $\mathcal{C}$ be a
  symmetric monoidal bicategory. Consider the $SO(2)$-action
  $\rho^{S}$ on $\core{\cC^\fd}$ induced by the
  Serre automorphism. Then the evaluation morphism $ev_{L}$ induces an
  equivalence of bicategories
  \begin{equation}
    \Fun_\ot(\F_{cfd} ,\cC)^{SO(2)} \to \core{\cC^\fd}^{SO(2)}.
  \end{equation}
\end{cor}
\begin{proof}
  By Proposition \ref{prop:serre-comes-from-fcfd}, the equivalence of
  Equation \eqref{eq:framed-cob-hyp-2} is $SO(2)$-equivariant. Thus,
  it induces an equivalence on homotopy fixed points,
  cf. \cite[Definition 5.3]{hesse16} for an explicit description. It
  is also possible to construct this equivalence directly: by theorem
  \ref{thm:fixed-points-non-triv-so2-action}, the bicategory of
  homotopy fixed points $\Fun_\ot(\F_{cfd},\cC)^{SO(2)}$ is equivalent
  to the bicategory where
  \begin{itemize}
  \item objects are given by symmetric monoidal functors $Z:\F_{cfd}
    \to \cC$, together with a modification $\lambda_Z: \tilde \alpha_Z
    \to \id_Z$. Explicitly, this means: if $\alpha$ is the
    endotransformation of the identity functor of $\F_{cfd}$ as in
    Definition \ref{def:action-framed-bordism-bicat}, we obtain two
    2-isomorphisms in $\cC$:
    \begin{equation}
      \begin{aligned}
        \lambda_L & :Z(q^{-1}) \to \id_{Z(L)} \\
        \lambda_R & :Z({({(q^{-1})}^*)}^{-1}) \to \id_{Z(R)}
      \end{aligned}
    \end{equation}
    which are compatible with evaluation and coevaluation,
  \item 1-morphisms are given by symmetric monoidal pseudo-natural
    transformations $\mu: Z \to Z'$, so that the analogue of the
    diagram in Equation \eqref{eq:better-thm-condition-1mor} commutes,
  \item 2-morphisms are given by symmetric monoidal modifications.
  \end{itemize}
  Now notice that $Z(q)$ is precisely the Serre automorphism of the
  object $Z(L)$. Thus, $\lambda_L$ provides a trivialization of (the
  inverse of) the Serre automorphism.  Applying theorem
  \ref{thm:fixed-points-non-triv-so2-action} again to the action of
  the Serre automorphism on the core of fully-dualizable objects shows
  that the functor $Z \mapsto (Z(L), \lambda_L)$ is an equivalence of
  homotopy fixed point bicategories.
\end{proof}
\begin{remark}
  Note that in Corollary \ref{cor:equivalence-fixed-point-bicats} we
  have proven that the evaluation induces an equivalence of
  bicategories
  $ \Fun_\ot(\F_{cfd} ,\cC)^{SO(2)} \to \core{\cC^\fd}^{SO(2)}$. We
  expect that this equivalence is an equivalence of \emph{monoidal}
  bicategories. In order to prove this, one would have to explicitly
  work out the monoidal structure of $ \core{\cC^\fd}^{SO(2)}$ which
  is induced from the monoidal structure of $\core{\cC^\fd}$.
\end{remark}

\section{Invertible Field Theories}
\label{sec:inv}
In the section, we consider the case of 2-dimensional oriented
\emph{invertible} topological field theories: such theories are in
many ways easier to describe than arbitrary TQFTs, and play an
important role in condensed matter physics and homotopy theory, as
suggested in \cite{Freed:2014iua, Freed:2014eja}.

Denote with $\Pic(\cC)$ the \emph{Picard groupoid} of a
symmetric monoidal bicategory $\cC$: it is defined as the maximal
subgroupoid of $\cC$ where the objects are invertible with respect to
the monoidal structure of $\cC$. Notice that $\Pic(\cC)$ inherits the symmetric monoidal structure from $\cC$. Recall that $\Fun_\ot(\Cob_{2,1,0},
\cC)$ is equipped with a monoidal structure which is defined
pointwise.

\begin{newdef}
  An invertible framed TQFT with values in $\cC$ is an
  invertible object in $\Fun_\ot(\Cob_{2,1,0}^{\fr}, \cC)$. The
  space of invertible framed TQFTs with values in $\cC$ is
  given by $\Pic(\Fun_\ot(\Cob_{2,1,0}, \cC))$.
\end{newdef} 
\begin{remark}
  Equivalently, an invertible TQFT assigns to the point in
  $\Cob_{2,1,0}$ an invertible object in $\cC$, and to any 1- and
  2-dimensional manifold it assigns invertible 1- and 2-morphisms.
\end{remark}
Since the Cobordism Hypothesis provides a \emph{monoidal} equivalence
between $\Fun_\ot(\Cob_{2,1,0}, \cC)$ and $\core{\cC^{\fd}}$, the
space of invertible framed TQFTs is given by
$\Pic(\core{\cC^{\fd}})$, since taking the Picard groupoid behaves
  well with respect to monoidal equivalences.

We begin by proving the following:
\begin{lemma}
\label{lem:pic-core-fd}
  Let $\cC$ be a symmetric monoidal bicategory. Then, there is an
  equivalence of symmetric monoidal bicategories
  \begin{equation}
    \Pic(\core{\cC^{\fd}})\cong \Pic(\cC).
  \end{equation}
\end{lemma}
\begin{proof}
  First note that $\core{\cC^\fd}$ is a monoidal 2-groupoid, so there
  is an equivalence of bicategories $\Pic(\core{\cC^\fd}) \cong
  \Pic(\cC^\fd)$. Now, it suffices to show that every object $X$ in
  $\Pic(\cC)$ is already fully-dualizable. Indeed, denote the
  tensor-inverse of $X$ by $X^{-1}$. By definition, we have
  1-equivalences $X \ot X^{-1} \cong 1$ and $ 1 \cong X^{-1} \ot X$,
  which serve as evaluation and coevaluation. These maps may be
  promoted to adjoint 1-equivalences by \cite[Proposition
  A.27]{schommerpries-classification}. Thus, the evaluation and
  coevaluation also admit adjoints, which suffices for fully-dualizability.
\end{proof}
Notice that given a monoidal bicategory $\cC$, any monoidal
auto-equivalence of $\cC$ preserves the Picard groupoid of $\cC$, since
it preserves invertibility of objects and (higher) morphisms. In particular, we have a
monoidal 2-functor
\begin{equation}
\Aut_\ot(\cC) \to \Aut_\ot(\Pic(\cC)) 
\end{equation}
obtained by restriction. Since the $SO(2)$-action induced by the
action on $\Cob_{2,1,0}$ is monoidal, it induces an action on
$\Pic(\cC)$.  To proceed, we need the following
\begin{lemma}
  \label{lemma:picard}
  Let $\cC$ be a symmetric monoidal bicategory such that $\Pic(\cC)$ is
  monoidally equivalent to $B^{2}\mathbb{K}^{*}$. Then
\begin{equation}
  \Aut_\ot(\Pic(\cC)) \simeq \mathrm{Iso}(\mathbb{K}^{*})
\end{equation}
where the category on the right hand side is regarded as a discrete symmetric monoidal bicategory.
\end{lemma}
\begin{proof}
 Since $\Pic(\cC)\simeq{B^{2}\mathbb{K}^*}$ monoidally, we have to describe the
 Picard groupoid of the category of monoidal functors from
 $B^{2}\mathbb{K}^*$ to $B^2\mathbb{K}^*$.
 First, notice that the monoidal bicategory $B^2\K^*$ is the strict symmetric
 monoidal bicategory with a single object $\bullet$, and $B\K^{*}$ as
 the strict symmetric monoidal category of 1- and 2-morphisms. The bicategory
 of symmetric monoidal functors from $B^2\K^*$ to itself is then
 equivalent to the category $\Fun_{\otimes}(B\K^{*}, B\K^{*})$
 regarded as a bicategory with only identity 2-cells; see
 \cite{chenggurski} for details. 
 
 By direct investigation, $\Fun_{\otimes}(B\K^{*}, B\K^{*})$ is
 equivalent as a symmetric monoidal category to $\Hom(\K^*,\K^*)$
 regarded as a discrete category. Indeed, any monoidal functor
 $F:B\K^{*}\to B\K^{*}$ is determined by a group homomorphism
 $\phi^{F}:\K^{*}\to \K^{*}$, and monoidality ensures that any natural
 transformation must correspond to the identity element in
 $\K^{*}$. Notice that the composition of monoidal functors
 $F^{'}\circ F$ corresponds to $\phi^{F'}\circ\phi^{F}$. In
 follows then that the Picard groupoid of
 $\Fun_{\otimes}(B^{2}\K^{*}, B^{2}\K^{*})$ is given by
 $\mathrm{Iso}(\mathbb{K}^{*})$, which correspond to the invertible
 elements in the monoid $\Hom(\K^*,\K^*)$.
\end{proof}
Examples of symmetric monoidal bicategories satisfying the assumption
of Lemma \ref{lemma:picard} are $\Alg_{2}^\fd$ and $\Vect_2^\fd$. In general cases, we have the following 
\begin{lemma}
  Let $\cC$ be a symmetric monoidal bicategory such that $\Pic(\cC)$
  is monoidally equivalent to $B^{2}\mathbb{K}^{*}$. Then any monoidal $SO(2)$-action on $\Pic(\cC)$ is trivializable.
\end{lemma}
\begin{proof}
Since we have monoidal equivalences $\Pi_{2}(SO(2))\simeq B\Z$ and
  $\Aut_\ot(\Pic(\cC))\simeq \textrm{Iso}(\mathbb{K}^{*})$, monoidal actions
  correspond to monoidal 2-functors $B \Z \to
  \textrm{Iso}(\mathbb{K}^{*})$: here we regard $B \Z$ as a symmetric monoidal bicategory with a single object, and the group $\mathrm{Iso}(\mathbb{K}^{*})$ as a discrete symmetric monoidal bicategory, i.e. all 1- and 2-cells are identities. Monoidality implies that the single object of
  $B \Z$ is sent to the identity isomorphism of $\mathbb{K}^*$, which
  correspond to the identity functor on $\Pic(\cC)$. This forces the
  functor to be trivial on objects. It is clear that the action is
  also trivial on 1- and 2-morphisms. Since there are no nontrivial
  morphisms in $\textrm{Iso}(\mathbb{K}^{*})$, the monoidal structure on the
  action $\rho$ must also be trivial.
\end{proof}
Finally, we need the following 
\begin{lemma}
  Let $\cC$ be a symmetric monoidal bicategory, and let $\rho_S$ be
  the $SO(2)$-action on $\core{\cC^\fd}$ by the Serre automorphism as
  in Example \ref{ex:serre-action-core-c-fd}. Since this action is
  monoidal, it induces an action on $\Pic(\core{\cC^\fd}) \cong
  \Pic(\cC)$ by Lemma \ref{lem:pic-core-fd}. We have then an
  equivalence of monoidal bicategories
  \begin{equation}
    \Pic\left((\core{\cC^\fd})^{SO(2)}\right) \cong \Pic(\cC)^{SO(2)}.
  \end{equation}
\end{lemma}
\begin{proof}
  Theorem \ref{thm:fixed-points-non-triv-so2-action} allows us to
  compute the two bicategories of homotopy fixed points explicitly: we
  see that both bicategories have invertible objects $X$ of $\cC$,
  together with the choice of a trivialization of the Serre
  automorphism as objects. The 1-morphisms of both bicategories are
  given by 1-equivalences between invertible objects of $\cC$, so that
  the diagram in equation \eqref{eq:serre-auto-fixed-point-condition}
  commutes, while 2-morphisms are given by 2-isomorphisms in $\cC$.
\end{proof}
The implication of the above lemmas is the following: when $\cC$ is a
symmetric monoidal bicategory with $\Pic(\cC) \cong B^2\K^*$, the
action of the Serre-automorphism on framed, invertible field theories
with values in $\cC$ is trivializable. Thus \emph{all} framed
invertible 2d TQFTs with values in $\cC$ can be turned into orientable
ones.
\section{Comments on Homotopy Orbits}
\label{sec:comments}
So far, we have constructed an $SO(2)$-action on the bicategory
$\F_{cfd}$. We have shown how the action on $\F_{cfd}$ induces an
action on the bicategory of symmetric monoidal functors
$\Fun_\ot(\F_{cfd}, \cC)$, and that via the (framed) Cobordism
Hypothesis the induced action on $\core{\cC^\fd}$ for framed manifolds
agrees with the action of the Serre automorphism. As a consequence, we
are able to provide an equivalence of bicategories
\begin{equation}\label{eq:fixed}
  \Fun_\ot(\F_{cfd} ,\cC)^{SO(2)} \to \core{\cC^\fd}^{SO(2)}
\end{equation}
in Corollary \ref{cor:equivalence-fixed-point-bicats}. We could then
in principle deduce the Cobordism Hypothesis for oriented manifolds
from \ref{eq:fixed}, once we provide an equivalence of bicategories
\begin{equation}
  \label{eq:suffices-to-show-for-oriented-cob-hyp}
  \Fun_\ot(\F_{cfd} ,\cC)^{SO(2)} \cong \Fun_\ot(\Cob_{2,1,0}^\ori, \cC).
\end{equation}
The above equivalence can be proven directly by using a presentation
of the oriented bordism bicategory via generators and relations, given
in \cite{schommerpries-classification}, and the notion of a Calabi-Yau
object internal to a bicategory. The details appear in
\cite{hesse-thesis}.

Here, we want instead to comment on an alternative approach. Namely, in
order to provide an equivalence as in
\eqref{eq:suffices-to-show-for-oriented-cob-hyp}, it suffices to
identify the oriented bordism bicategory with the \emph{colimit} of
the $SO(2)$-action on $\F_{cfd}$. Indeed, recall that one may define a
$G$-action on a bicategory $\cC$ to be a trifunctor $\rho: B \Pi_2(G)
\to \Bicat$ with $\rho(*)=\cC$. The tricategorical colimit of this
functor will then be the bicategory of co-invariants or \emph{homotopy
  orbits} of the $G$-action, denoted by $\cC_G$. By Definition of the
tricategorical colimit, and the fact that colimits are sent to limits
by the $\textrm{Hom}$ functor, we then obtain an equivalence of
bicategories
\begin{equation}\label{eq:colim-induced}
  \Fun_\ot(\cC_G ,\cD) \cong \Fun_\ot(\cC,\cD)^G.
\end{equation}
The following conjecture is then natural:
\begin{conj}
  \label{conj:colim}
  The bicategory of co-invariants of the $SO(2)$-action on $\F_{cfd}$
  is monoidally equivalent to the oriented bordism bicategory, i.e. we have a monoidal equivalence
  \begin{equation}
    ( \F_{cfd})_{SO(2)} \cong \Cob_{2,1,0}^\ori .
  \end{equation}
  Furthermore, the colimit is compatible with the monoidal structure.
\end{conj}
\begin{remark}
  We believe that this is not an isolated phenomenon, in the sense
  that any higher bordism category equipped with additional tangential
  structure should be obtained by taking an appropriate colimit of a
  $G$-action on the framed bordism category.
\end{remark}
Given Conjecture \ref{conj:colim} and Equation \ref{eq:colim-induced}, we obtain the following sequence of monoidal equivalences
\begin{equation}\label{eq:chain}
\Fun_\ot(\Cob_{2,1,0}^\ori, \cC)\cong\Fun_\ot(( \F_{cfd})_{SO(2)}, \cC)\cong\Fun_\ot(\F_{cfd} ,\cC)^{SO(2)}\cong\core{\cC^\fd}^{SO(2)}.
\end{equation}
Hence Conjecture \ref{conj:colim} implies the Cobordism Hypothesis for oriented 2-manifolds. Notice that the chain of equivalences in \ref{eq:chain} is natural in $\mathcal{C}$.\\
On the other hand, the Cobordism Hypothesis for oriented manifolds in
2-dimensions implies Conjecture \ref{conj:colim} . Indeed, by using a
tricategorical version of the Yoneda Lemma, as developed for instance
in \cite{buhne15}, the chain of equivalences
\begin{align}
\begin{split}
\Fun_\ot(\Cob_{2,1,0}^\ori, \cC) & \cong \core{\cC^\fd}^{SO(2)}\\
& \cong \Fun_\ot(\Cob_{2,1,0}^\fr, \cC)^{SO(2)}\\ 
& \cong \Fun_\ot(( \F_{cfd})_{SO(2)}, \mathcal{C})
\end{split}
\end{align}
implies that $\Cob_{2,1,0}^\ori$ is equivalent to $( \F_{cfd})_{SO(2)}$, due to the uniqueness of representable objects.\\
We summarize the above arguments in the following 
\begin{lemma}
The Cobordism Hypothesis for oriented 2-dimensional manifolds is equivalent to Conjecture \ref{conj:colim}.
\end{lemma}
It would then be of great interest to develop concrete constructions
of homotopy co-invariants of actions of groups on bicategories, in the
same spirit of \cite{hsv16} and the present work, in order to verify
directly the equivalence in Conjecture \ref{conj:colim}, and to extend the above arguments to general tangential $G$-structures.

\begin{landscape}
  \vspace*{\fill}
  \begin{figure}[htbp]
    \includegraphics[width=\linewidth]{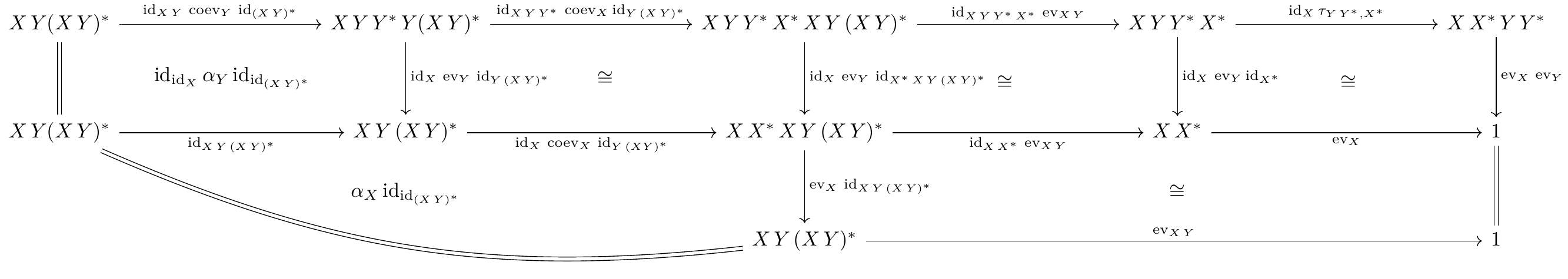}
    \caption{Diagram for the proof of Lemma \ref{lem:ev-monoidal}}
    \label{fig:ev-monoidal}
  \end{figure}
  \vspace*{\fill}
\end{landscape}

\providecommand{\bysame}{\leavevmode\hbox to3em{\hrulefill}\thinspace}
\providecommand{\MR}{\relax\ifhmode\unskip\space\fi MR }
\providecommand{\MRhref}[2]{%
  \href{http://www.ams.org/mathscinet-getitem?mr=#1}{#2}
}
\providecommand{\href}[2]{#2}

\end{document}